\documentclass[11pt]{amsart}

\usepackage[utf8]{inputenc}
\usepackage{ulem}
\usepackage{authblk}
\usepackage[english]{babel}
\usepackage{amssymb}
\usepackage[margin=1in]{geometry}
\usepackage{verbatim}
\usepackage{amsthm}
\usepackage{amsmath}
\usepackage{amssymb}
\usepackage{amsaddr}
\usepackage{graphicx}
\usepackage[colorinlistoftodos]{todonotes}
\usepackage[colorlinks=true, allcolors=blue]{hyperref}
\usepackage{comment}
\usepackage{verbatim}
\usepackage{ulem}
\usepackage{cleveref}  
\usepackage{enumitem}

\usepackage[numbers]{natbib}

\newcommand{\E}{\mathrm{E}}
\newcommand{\Var}{\mathrm{Var}}
\newcommand{\Cov}{\mathrm{Cov}} 
\newcommand{\alpren}{$\alpha$-R\'enyi }
\newcommand{\loglik}{r_n(\theta,\theta_0)}
\newcommand{\malpharen}{D_{\alpha}(P^{(n)}_{\theta},P^{(n)}_{\theta_0})}
\newcommand{\data}{X_0,\dots,X_n}
\newcommand{\datas}{x_0,\dots,x_n}

\newcommand{\kld}{ \mathcal{K}(P^{(n)}_{\theta_0},P^{(n)}_{\theta})}
\newcommand{\lowersigma}{\mathcal{M}_{-\infty}^{t}}
\newcommand{\uppersigma}{ \mathcal{M}_{t+k}^{\infty}}

\newcommand{\psirr}{$\psi$-irreducible }

\newtheorem{defn}{Definition}[section]

\newtheorem{corollary}{Corollary}[section]
\newtheorem{theorem}{Theorem}[section]
\newtheorem{lemma}{Lemma}[section]
\newtheorem{proposition}{Proposition}[section]
\newtheorem{assumption}{Assumption}[section]
\newtheorem{example}{Example}[section]
\newtheorem*{lemma*}{Lemma}

\newcommand{\ifix}[1]{}
\newcommand{\commentfixed}[1]{ }

\newcommand{\hh}[1]{}

\begin{document}
\title[PAC-Bayes for Markov Models]{PAC-Bayes Bounds on Variational Tempered Posteriors for Markov Models}
\author{Imon Banerjee, Vinayak A. Rao}
\address{Department of Statistics, Purdue University, West Lafayette IN 47906}
\email{\{ibanerj,varao,honnappa\}@purdue.edu}

\author{Harsha Honnappa}
\address{School of Industrial Engineering, Purdue University, West Lafayette IN 47906}

\date{}
\maketitle
\begin{abstract}
    
    Datasets displaying temporal dependencies abound in science and engineering applications, with
    Markov models representing a simplified and popular view of the temporal dependence structure. 
    In this paper, we consider Bayesian settings that place prior distributions over the parameters of the transition kernel of a Markov model, and seeks to characterize the resulting, typically intractable, posterior distributions.
    We present a PAC-Bayesian analysis of variational Bayes (VB) approximations to tempered Bayesian posterior distributions, bounding the model risk of the VB approximations.
    Tempered posteriors are known to be robust to model misspecification, and their variational approximations do not suffer the usual problems of over confident approximations.
    Our results tie the risk bounds to the mixing and ergodic properties of the Markov data generating model. We illustrate the PAC-Bayes bounds through a number of example Markov models, and also consider the situation where the Markov model is misspecified.
\end{abstract}
\section{Introduction}
This paper presents probably approximately correct (PAC)-Bayesian bounds on variational Bayesian (VB) approximations of fractional or tempered posterior distributions for Markov data generation models. Exact computation of either standard or tempered posterior distributions is a hard problem that has, broadly speaking, spawned two classes of computational methods. The first, Markov chain Monte Carlo (MCMC), constructs ergodic Markov chains to approximately sample from the posterior distribution. MCMC is known to suffer from high variance and complex diagnostics, leading to the development of variational Bayesian (VB)~\cite{jordan} methods as an alternative in recent years. 
VB methods pose posterior computation as a variational optimization problem, approximating the posterior distribution of interest by the `closest' element of an appropriately defined class of `simple' probability measures. 
Typically, the measure of closeness used by VB methods is the Kullback-Leibler (KL) divergence. 
Excellent introductions to this so-called {\it KL-VB} method can be found in \cite{bishop,ormerod,blei}. More recently, there has also been interest in alternative divergence measures, particularly the $\alpha$-R\'enyi divergence~\cite{prateek2,li2016renyi, dieng2017variational}, though in this paper, we focus on the KL-VB setting.

 Theoretical properties of VB approximations, and in particular asymptotic frequentist consistency, have been studied extensively under the assumption of an independent and identically distributed (i.i.d.) data generation model~\cite{blei,blei1,zhang}. On the other hand, the common setting where data sets display temporal dependencies presents unique challenges. In this paper, we focus on homogeneous Markov chains with parameterized transition kernels, representing a parsimonious class of data generation models with a wide range of applications.  
We work in the Bayesian framework, focusing on the posterior distribution over the unknown parameters of the transition kernel.
Our theory develops PAC bounds that link the ergodic and mixing properties of the data generating Markov chain to the Bayes risk associated with approximate posterior distributions.

Frequentist consistency of Bayesian methods, in the sense of concentration of the posterior distribution around neighborhoods of the `true' data generating distribution, have been established in significant generality, in both the i.i.d.\ \cite{ghosh,shen,rous} and in the non-i.i.d.\ data generation setting~\cite{ghoshal,pati}.
More recent work~\cite{pierre,bhatta,pati} has studied {fractional or tempered posteriors}, a class of generalized Bayesian posteriors obtained by combining the likelihood function raised to a fractional power with an appropriate prior distribution using Bayes theorem.
Tempered posteriors are known to be robust against model misspecification: in the Markov setting we consider, the associated stationary distribution  as well as mixing properties are sensitive to model parameterization. 
Further, tempered posteriors are known to be much simpler to analyze theoretically~\cite{bhatta,pati}.
Therefore, following~\cite{pierre,bhatta,pati} we focus on tempered posterior distributions on the transition kernel parameters, and study the rate of concentration of variational approximations to the tempered posterior. 
Equivalently, as shown in~\cite{bhatta} and discussed in~\cref{sec:notation}, our results also apply to so-called $\alpha$-variational approximations to standard posterior distributions over kernel parameters.
The latter are modifications of the standard KL-VB algorithm to address the well-known problem of overconfident posterior approximations.

While there have been a number of recent papers studying the consistency of approximate variational posteriors~\cite{blei1,prateek2,pierre} in the large sample limit, rates of convergence have received less attention. Exceptions include~\cite{pierre,zhang,prateek1}, where an i.i.d. data generation model is assumed. \cite{pierre} establishes PAC-Bayes bounds on the convergence of a variational {tempered posterior} with fractional powers in the range $[0,1)$, while~\cite{zhang} considers the standard variational posterior case (where the fractional power equals $1$).~\cite{prateek1}, on the other hand, establishes PAC-Bayes bounds for risk-sensitive Bayesian decision making problems in the standard variational posterior setting. The setting in~\cite{pierre} allows for model misspecification and the analysis is generally more straightforward than that in~\cite{zhang,prateek1}. Our work extends~\cite{pierre} to the setting of a discrete-time Markov data generation model. 

Our first results in~\Cref{thm:pac-bayes} and~\Cref{thm:pac-bayes2} of \Cref{section:conc bound} establish PAC-Bayes bounds for sequences with arbitrary temporal dependence. Our results generalize \cite[Theorem 2.4]{pierre} to the non-i.i.d.\ data setting in a straightforward manner. Note that~\Cref{thm:pac-bayes} also recovers~\cite[Theorem 3.3]{bhatta}, which is established under different `existence of test' conditions. Our objective in this paper is to explicate how the ergodic and mixing properties of the Markov data generating process influences the PAC-Bayes bound. The sufficient conditions of our theorem, bounding the mean and variance of the log-likelihood ratio of the data, allows for developing this understanding, without the technicalities of proving the existence of test conditions intruding on the insights. 

In Section~\ref{sec:sedgm} we study the setting where the data generating model is a stationary $\alpha$-mixing Markov chain. Stationarity implies the existence of an invariant distribution corresponding to the parameterized transition kernel, implying the marginal distribution of the Markov data is invariant as well. The $\alpha$-mixing condition, on the other hand, ensures that the variance of the likelihood ratio of the Markov data does not grow faster than linear in the sample size. Our main results in this setting are applicable when the state space of the Markov chain is either continuous or discrete. The primary requirement on the class of data generating Markov models is for the log-likelihood ratio of the parameterized transition kernel and invariant distribution to satisfy a Lipschitz property. This condition implies a decoupling between the model parameters and the random samples, affording a straightforward verification of the mean and variance bounds. We highlight this main result by demonstrating that it is satisfied by a finite state Markov chain, a birth-death Markov chain on the positive integers, and a one-dimensional Gaussian linear model. 

In practice, the assumption that the data generating model is stationary is unlikely to be satisfied. 
Typically, the initial distribution is arbitrary, with the state distribution of the Markov sequence converging weakly to the stationary distribution.
In this setting, we must further assume that the class of data generating Markov chains are geometrically ergodic (with an exception for finite state Markov chains). 
We show that this implies the boundedness of the mean and variance of the log-likelihood ratio of the data generating Markov chain. Alternatively, in~\Cref{thm:gen-nonstationary-hajek} we directly impose a drift condition on random variables that bound the log-likelihood ratio. Again, in this more general nonstationary setting, we illustrate the main results by showing that the PAC-Bayes bound is satisfied by a finite state Markov chain, a birth-death Markov chain on the positive integers, and a one-dimensional Gaussian linear model.

In preparation for our main technical results starting in~\Cref{section:conc bound} we first note relevant notations and definitions in the next section.

\subsection{Notations and Definitions}~\label{sec:notation}
We broadly adopt the notation in~\cite{pierre}. Let the sequence of random variables $X^n = (X_0,\ldots,X_n) \subset \mathbb R^{m\times (n+1)}$ represent a data set of $n+1$ observations drawn from a joint distribution $P_{\theta_0}^{(n)}$, where $\theta_0 \in \Theta \subseteq \mathbb R^d$ is the `true' parameter underlying the data generation process. 
We assume the state space $S \subseteq \mathbb R^m$ of the random variables $X_i$ is either discrete-valued or continuous, and write $\{\datas\}$ for a realization of the dataset. We also adopt the convention that $0 \log (0/0) = 0$.

For each $\theta \in \Theta$, we will write $p_{\theta}^{(n)}$ as the probability density of $P_{\theta}^{(n)}$ with respect to some measure $Q^{(n)}$, i.e. $p_{\theta}^{(n)}:=\frac{dP_{\theta}^{(n)}}{dQ^{(n)}}$, where $Q^{(n)}$ is either Lebesgue measure or the counting measure.
All expectations and variances, which we represent as $\E[X]$ and $\Var[X]$, are taken with respect to the true distribution $P_{\theta_0}$ unless stated otherwise.

Let $\pi(\theta)$ be a {\it prior} distribution with support $\Theta$. The {\it fractional posterior} is defined as 
\begin{align}
    \pi_{n,\alpha|X^n} (d\theta)& := \frac{ e^{-\alpha\loglik(X^n)} \pi(d\theta)}{ \int e^{-\alpha\loglik(X^n)} \pi(d\theta) },
  \end{align}
where, for $\theta, \theta_0 \in \Theta$, $r_n(\theta,\theta_0)(\cdot) := \log\left(\frac{p^{(n)}_{\theta_0}(\cdot)}{p^{(n)}_{\theta}(\cdot)}\right)$, is the log-likelihood ratio of the corresponding density functions. Setting $\alpha=1$ recovers the standard Bayesian posterior. 

The\textit{ Kullback-Leibler} (KL) divergence between distributions $P,Q$ is defined as
\begin{align*}
    \mathcal{K}(P,Q) := \int_{\mathcal{X}} \log\left( \frac{p(x)}{q(x)}\right) p(x) dx,
\end{align*}
where $\mathcal{X}$ is an arbitrary sample space, and $p,q$ are the densities corresponding to $P,Q$ (respectively). In particular, the KL divergence between the distributions parameterized by $\theta_0$ and $\theta$ is
\begin{align}
    \nonumber
    \kld & := \int \log \left( \frac{p^{(n)}_{\theta_0}(\datas)}{p^{(n)}_{\theta}(\datas)} \right) p^{(n)}_{\theta_0}(\datas) dx_0 \cdots dx_n\\
    &= \int r_n(\theta,\theta_0)(\datas) p^n_{\theta_0}(\datas) dx_0 \cdots dx_n.
 \end{align}
The\textit{ $\alpha$-R\'enyi divergence} $\malpharen$ is defined as 
\begin{align}
    \malpharen & := 
    \frac{1}{\alpha-1} \log \int \exp\left(-\alpha r_n(\theta,\theta_0)(\datas)\right) p^{(n)}_{\theta_0}(\datas) dx_0 \cdots dx_n.
\end{align}
By letting $\alpha\to 0$, the \alpren divergence recovers the KL divergence. 

Let $\mathcal F$ be some class of distributions with support in $\mathbb R^d$ and such that any distribution $P$ in $\mathcal{F}$ is absolutely continuous with respect to the tempered posterior:
 $P \ll \pi_{n,\alpha|X^n}$. 
 Let $\Tilde{\pi}_{n,\alpha|X^n}$ be the variational approximation to the tempered posterior, defined as 
\begin{align}~\label{eq:vba}
    \Tilde{\pi}_{n,\alpha|X^n} & := \underset{\rho\in \mathcal{F}}{\arg \min} \enspace \mathcal{K}(\rho,\pi_{n,\alpha|X^n})
\end{align}
Many choices of $\mathcal{F}$ exist; for instance, $\mathcal{F}$ can be the set of Gaussian measures, denoted $\mathcal{F}^{\Phi}_{id}$:
\begin{align}
	 \mathcal{F}^{\Phi}_{id} & = \{\Phi(d\theta;\mu,\Sigma): \mu\in \mathbb{R}^{d}, \Sigma_{d\times d} \in \text{P.D.}  \},
\end{align}
where $\text{P.D.}$ references the class of positive definite matrices. 
Alternately, $\mathcal{F}$ can be the family of {\em mean-field} or {factored distributions} where the components $\theta_i$ of $\theta$ are independent of each other. 

It is easy to see that \cref{eq:vba} is equivalent to the following optimization problem:
\begin{align}~\label{eq:elbo}
    \Tilde{\pi}_{n,\alpha|X^n} & := \underset{\rho\in \mathcal{F}}{\arg \max}  \int r_n(\theta,\theta_0)(\datas) \rho({d\theta}) - \alpha^{-1} \mathcal{K}(\rho,\pi).
\end{align}
Setting $\alpha=1$ again recovers the usual variational solution that seeks to approximate the posterior distribution with the closest element of $\mathcal{F}$ (the right-hand side above is called the evidence lower bound (ELBO)).
Other settings of $\alpha$ constitute $\alpha$-variational inference~\citep{bhatta}, which seeks to regularize the `overconfident' approximate posteriors that standard variational methods tend to produce. 

\subsubsection{Markov chains}~\label{section:ergodicity}

We assume the joint density or probability mass function $p_{\theta}^{(n)}(\datas)$ corresponds to the `walk probability' of a time-homogeneous Markov chain. We call the Markov chain `parameterized' if the transition kernel $p_\theta(\cdot|\cdot)$ is parametrized by some $\theta \in \Theta \subseteq \mathbb R^d$. Let $q^{(0)}(\cdot)$ be the initial density (defined with respect to the Lebesgue measure over $\mathbb R^m$) or initial probability mass function. Then, the joint density or probability mass function is $p_{\theta}^{(n)}(\datas) = 
q^{(0)}(x_0) \prod_{i=0}^{n-1} p_{\theta}(x_{i+1}|x_i)$.

Our results in the ensuing sections will be established under strong mixing conditions~\cite{bradley} on the Markov chain. Specifically, recall the definition of the $\alpha$-mixing coefficients of a stationary Markov chain:
\begin{defn}[{$\alpha$-mixing coefficient}] 
Let $\mathcal{M}_i^j$ denote the $\sigma$-field generated by the Markov chain $\{X_k : i\leq k\leq j\}$ parameterized by $\theta \in \Theta$. Then, the $\alpha$-mixing coefficient is defined as    
\begin{align}
    \alpha_k = \underset{t>0}{\sup}\underset{(A,B)\in \mathcal{M}_{-\infty}^t \times \mathcal{M}_{t+k}^{\infty} }{\sup} \left|P_{\theta}(A\cap B)-P_{\theta}(A)P_{\theta}(B)\right|.
\end{align}
\end{defn}
Informally speaking, the $\alpha$-mixing coefficients $\{\alpha_k\}$ measure the dependence between any two events $A$ (in the `history' $\sigma$-algebra) and $B$ (in the `future' $\sigma$-algebra) with a time lag $k$. Our results in Section~\ref{sec:nsedgm} also rely on the ergodic properties of the Markov chain, and we assume that the Markov chain is $V$-geometrically ergodic \cite[Chapter 15]{meyntweedie}. First, refer to the definition of the functional norm $\|\cdot\|_V$, from \Cref{define:v-norm},

\begin{defn}[{$f$-norm}]~\label{define:v-norm}The functional norm in $f$-metric of the measure $v$, or the {\it $f$-norm} of $v$ is defined as
\begin{equation}
    \|v\|_f=\underset{g:|g|<f}{\sup} \left|\int g dv \right|.
\end{equation}
\end{defn}

 An immediate consequence of this definition is that if $f_1, f_2$ are two functions such that $f_1 < f_2$ (i.e., for all points in the support of the functions), then
\begin{equation}~\label{equation:f-norm-monotone}
    \|v\|_{f_1}\leq\|v\|_{f_2}.
\end{equation} 

Now that we have defined the $\|\cdot\|_f$ norm, we can now define $V$-geometric ergodicity. In the following, we assume the Markov chain is positive Harris; see~\cite{meyntweedie} for a definition. This is a mild and fairly standard assumption in Markov chain theory.

\begin{defn}[{$V$-geometric ergodicity}]~\label{define:v-ergodicity}A stationary Markov chain $\{X_n\}$ parameterized by $\theta \in \Theta$ is $V$-geometrically ergodic if it is positive Harris and there exists a constant $r_V > 1$, that depends on $V$, such that for any $A \in \mathcal B(X)$,
\begin{equation}
    \sum_{n=1}^n r_V^n \left\|P_\theta(X_n \in A | X_0 = x)- \int_A q_\theta(y) dy \right\|_V < \infty.
\end{equation}
\end{defn}
It is straightforward to see that this is equivalent to 
\begin{equation*}
    \left\|P_\theta(X_n \in A |X_0=x)-\int q_\theta(y) dy \right\|_V\leq C r^{-n}
\end{equation*}
for an appropriate constant $C$ (which may depend on the state $x$). That is, the Markov chain approaches steady state at a geometrically fast rate.
If a Markov chain is $V$-geometrically ergodic for $V\equiv1$, then, it is simply termed as \textit{geometrically ergodic}. It is straightforward to see (via \cref{theorem:jones} in the appendix) that a geometrically ergodic Markov chain is also $\alpha$-mixing, with $\alpha$ coefficients satisfying
\begin{equation}\label{eq:alpha_mixing-coeff-sum}
    \sum_{k\geq 0} \alpha_k^{\upsilon}<\infty\enspace\forall\enspace\upsilon>0,
\end{equation}
{showing that, under geometric ergodicity, the alpha mixing coefficients raised to any positive power $\upsilon$ are finitely summable.} We note here that the most standard procedure to establish $V$-geometric ergodicity for any Markov chain is through the verification of the drift condition, \Cref{assume:drift}. If a Markov chain satisfies \Cref{assume:drift}, then \Cref{theorem:v-geom} proves that it is also $V$-geometrically ergodic. 

\section{A Concentration Bound for the $\alpha$-R\'enyi Divergence}\label{section:conc bound}
The object of analysis in what follows is the probability measure $ \Tilde{\pi}_{n,\alpha|X^n}(\theta)$, the variational approximation to the tempered posterior.
Our main result establishes a bound on the Bayes risk of this distribution; in particular, given a sequence of loss functions $\ell_n(\theta,\theta_0)$, we bound $\int \ell_n(\theta,\theta_0) \Tilde{\pi}_{n,\alpha|X^n}(\theta) d\theta$.
Following recent work 
in both the i.i.d.\ and dependent sequence setting~\cite{pati,pierre,bhatta}, we will use $\ell_n(\theta,\theta_0) = D_{\alpha}(P^{(n)}_{\theta},P^{(n)}_{\theta_0})$, the $\alpha$-R\'enyi divergence between $P^{(n)}_\theta$ and $P^{(n)}_{\theta_0}$ as our loss function 
(recall that for each $\theta \in \Theta$ and $n \ge 1$, $P_\theta^{(n)}$ is the distribution corresponding to the sequence $\{X_0,\ldots,X_n\}$). 
Unlike more obvious loss functions like Euclidean distance, R\'enyi divergence compares $\theta$ and $\theta_0$ through their effect on observed sequences, so that issues like parameter identifiability are no longer relevant.
Our first result generalizes~\cite[Theorem 2.1]{pierre} to a non-i.i.d.\ data setting.  

\begin{proposition}~\label{prop:1}
Let $\mathcal{F}$ be a subset of all probability distributions on $\Theta$. For any $\alpha\in (0,1)$, $\epsilon \in (0,1)$ and $n \ge 1$, the following probabilistic uniform upper bound on the expected \alpren divergence holds:
\begin{align}
    P^{(n)}_{\theta_0} \left[ \sup_{\rho \in \mathcal F}\int D_{\alpha}(P^{(n)}_{\theta},P^{(n)}_{\theta_0})\rho(d\theta)\leq \frac{\alpha}{1-\alpha}\int r_n(\theta,\theta_0)\rho(d\theta) + \frac{\mathcal{K}(\rho,\pi) + \log(\frac{1}{\epsilon}) }{1-\alpha} \right]\geq 1-\epsilon.
\end{align}
\end{proposition}
The proof of Proposition~\ref{prop:1} follows easily from~\cite{pierre}, and we include it in~\Cref{appendix:prop:1} for completeness. Mirroring the comments in~\cite{pierre}, when $\rho = \tilde \pi_{n,\alpha}$ this result is precisely~\cite[Theorem 3.5]{pati}. This probabilistic bound implies the following PAC-Bayesian concentration bound on the model risk computed with respect to the fractional variational posterior:

\renewcommand{\theenumi}{\roman{enumi}}

\begin{theorem}\label{thm:pac-bayes}
Let $\mathcal{F}$ be a subset of all probability distributions parameterized by $\Theta$, and assume there exist $\epsilon_n >0$ and $\rho_n\in \mathcal{F}$ such that 
\begin{enumerate}[align=left,leftmargin=*]
    \item $\int \kld \rho_n(d\theta) = \int \E[r_n(\theta,\theta_0)] \rho_n(d\theta) \leq n\epsilon_n $,
    \item $\int \Var \left( r_n(\theta,\theta_0) \right)\rho_n (d\theta)\leq n\epsilon_n$, and
    \item $\mathcal{K}(\rho_n,\pi)\leq n\epsilon_n$.
\end{enumerate}
Then, for any $\alpha \in (0,1)$ and $(\epsilon, \eta) \in (0,1)\times(0,1)$,
\begin{align}~\label{eq:thm21}
    P\left[ \int  D_{\alpha}(P^{(n)}_{\theta},P^{(n)}_{\theta_0})\Tilde{\pi}_{n,\alpha}(d\theta|X^{(n)} )\leq \frac{ (\alpha+1)n\epsilon_n+\alpha\sqrt{\frac{n\epsilon_n}{\eta}}-\log(\epsilon) } {  1-\alpha}  \right]\geq 1-\epsilon-\eta.
\end{align}
\end{theorem}
The proof of Theorem~\ref{thm:pac-bayes} is a straightforward generalization of~\cite[Theorem 2.4]{pierre} to the non-i.i.d. setting,
and a special case of~\cite[Theorem 3.1]{bhatta}, where the problem setting includes latent variables. We include a proof for completeness. As noted in~\cite{pierre}, the sufficient conditions follow closely from \cite{ghoshal} and we will show that they hold for a variety of Markov chain models.

A direct corollary of Theorem~\ref{thm:pac-bayes} follows by setting $\eta = \frac{1}{n\epsilon_n}$, $\epsilon=e^{-n\epsilon_n}$ and using the fact that $e^{-n\epsilon_n} \ge\frac{1}{n\epsilon_n}$. Note that~\cref{eq:thm21} is vacuous if $\eta+\epsilon>1$. Therefore, without loss of generality, we restrict ourselves to the condition $\frac{2}{n\epsilon_n}<1$.
\begin{corollary}~\label{thm:pac-bayes2}
Assume $\exists \hspace{1 mm}\epsilon_n >0$, $\rho_n\in \mathcal{F}$ such that the following conditions hold:
\begin{enumerate}[leftmargin=*,align=left]
    \item $\int \kld \rho_n(d\theta) = \int E[r_n(\theta,\theta_0)] \rho_n(d\theta) \leq n\epsilon_n $~\label{assume:pac-bayes1},
    \item $\int \Var \left( r_n(\theta,\theta_0) \right)\rho_n (d\theta)\leq n\epsilon_n$~\label{assume:pac-bayes2}, and
    \item $\mathcal{K}(\rho_n,\pi)\leq n\epsilon_n$.~\label{assume:pac-bayes3}
\end{enumerate}
Then, for any $\alpha \in (0,1)$,
\begin{align}
    P\bigg[ \int  D_{\alpha}(P^{(n)}_{\theta},P^{(n)}_{\theta_0})\Tilde{\pi}_{n,\alpha}(d\theta|X^{(n)} )\leq \frac{ (\alpha+1 )\epsilon_n } {  1-\alpha}  \bigg]\geq 1-\frac{2}{n\epsilon_n}.
\end{align}
\end{corollary}

Observe that the first two conditions in~\Cref{thm:pac-bayes2} ensure that the distribution $\rho_n$ concentrates on parameters $\theta \in \Theta$ around the true parameter $\theta_0$, while the third condition requires that $\rho_n$ not diverge from the prior $\pi$ rapidly as a function of the sample size $n$.
In general, satisfying the first and third conditions is relatively straightforward. The second condition, on the other hand, is significantly more complicated in the current setting of dependent data, as the variance of $r_n(\theta,\theta_0)$ includes correlations between the observations $\{\data\}$. In the next section, we will make assumptions on the transition kernels (and corresponding invariant densities) that 'decouple' the temporal correlations and the model parameters in the setting of strongly mixing and ergodic Markov chain models, and allow for the verification of the conditions in~\Cref{thm:pac-bayes2}.

\.The computations critical to the verification of the conditions in \Cref{thm:pac-bayes2} are the bounds (\ref{assume:pac-bayes1}) and (\ref{assume:pac-bayes2}), and 
Propositions~\ref{prop:1a} and~\ref{prop:2} below characterize the expectation and variance of the log-likelihood ratio $r_n(\cdot,\cdot)$ in terms of the one-step transition kernels of the Markov chain.  First, consider the expectation of $r_n(\cdot,\cdot)$ in condition (\ref{assume:pac-bayes1}). 

\begin{proposition}~\label{prop:1a}
Fix $\theta_1,\theta_2 \in \Theta$ and consider the parameterized Markov transition kernels  $p_{\theta_1}$ and $p_{\theta_2}$, and initial distributions $q^{(0)}_{\theta_1}$ and $q^{(0)}_{\theta_2}$. Let $p_{\theta_1}^{(n)}$ and $p_{\theta_2}^{(n)}$ be the corresponding joint probability densities; that is,
\(
    p_{\theta_j}^{(n)}(x_0,\ldots,x_n) = q_{\theta_j}^{(0)}(x_0)\prod_{i=1}^{n} p_{\theta_i}(x_i|x_{i-1})
\)
for $j \in \{1,2\}$. Then, for any $n \geq 1$, the log-likelihood ratio $r_n(\theta_2,\theta_1)$ satisfies
\begin{align}~\label{eq:exp-ll}
     \E_{\theta_1}\left[ r_n(\theta_2,\theta_1) \right] 
     &= \sum_{i=1}^n \E_{\theta_1} \left[ \log \left( \frac{p_{\theta_1}(X_i|X_{i-1})}{p_{\theta_2}(X_i|X_{i-1})}\right) \right]+\E[Z_0],
\end{align}
where $Z_0 := \log\left(\frac{q^{(0)}_{\theta_1}(X_0)}{q^{(0)}_{\theta_2}(X_0)}\right)$. The expectation in the first term is with respect to the joint density function $p_{\theta_1}(y,x) = p_{\theta_1}(y|x) q^{(i-1)}_{\theta_1}(x)$ where the marginal density satisfies
\begin{align*}
q^{(i-1)}_{\theta_1}(x)=\begin{cases}
\int p_{\theta_1}^{(i-1)}(x_0,\ldots,x_{i-2},x) dx_{0}\cdots dx_{i-2}&~\text{for}~i > 1, and\\
 q_{\theta_1}^{(0)}(x)&~\text{for}~i = 1.
\end{cases}
\end{align*}
If the Markov chain is also stationary under $\theta_1$, then~\cref{eq:exp-ll} simplifies to
\begin{align}
    \E_{\theta_1}\left[ r_n(\theta_2,\theta_1) \right]=n\E_{\theta_1} \left[ \log\left( \frac{p_{\theta_1}(X_1|X_0)}{p_{\theta_2}(X_1|X_0)} \right) \right]+\E_{\theta_1}[Z_0].
\end{align}
\end{proposition}
Notice that $\E_{\theta_1}\left[ r_n(\theta_2,\theta_1) \right]$ is precisely the KL divergence, $\mathcal{K}(P^{(n)}_{\theta_1},P^{(n)}_{\theta_2})$. Next, the following proposition uses~\citep[Lemma 1.3]{ibrag} to upper bound the variance of the log-likelihood ratio.

\begin{proposition}~\label{prop:2}
Fix $\theta_1,\theta_2 \in \Theta$ and consider parameterized Markov transition kernels $p_{\theta_1}$ and $p_{\theta_2}$, with initial distributions $q^{(0)}_{\theta_1}$ and $q^{(0)}_{\theta_2}$. Let $p_{\theta_1}^{(n)}$ and $p_{\theta_2}^{(n)}$ be the corresponding joint probability densities of the sequence $(x_0,\dotsc,x_n)$, and $q^{(i)}_{\theta_j}$ the marginal density for $i \in \{1,\ldots,n\}$ and $j \in \{1,2\}$. Fix $\delta>0$ and, for each $i\in \{1,\dots,n\}$, define
\begin{align*}
C_{\theta_1,\theta_2}^{(i)} &:= \int\left|\log\left(\frac{ p_{\theta_1}(x_i|x_{i-1}) } { p_{\theta_2}(x_i|x_{i-1}) }\right)\right|^{2+\delta} p_{\theta_1}(x_i| x_{i-1}) q^{(i-1)}_{\theta_1}(x_{i-1}) dx_i dx_{i-1}.
\end{align*}
Similarly, define $Z_0 := \log\left(\frac{q^{(0)}_{\theta_1}(X_0)}{q^{(0)}_{\theta_2}(X_0)}\right)$, and
$    D_{1,2} := \E_{q_{\theta_1}^{(0)}}\left|Z_0\right|^{2+\delta}.$
Suppose the Markov chain corresponding to $\theta_1$ is $\alpha$-mixing with coefficients $\{\alpha_k\}$. Then,
\begin{align}~\label{eq:var-ll}
\nonumber
    \Var \left( r_n(\theta_1,\theta_2) \right) & < \sum_{i,j=1}^n\left(\frac{4}{n}+2n^{\delta/2}(C_{\theta_1,\theta_2}^{(i)}+C_{\theta_1,\theta_2}^{(j)}+2\sqrt{C_{\theta_1,\theta_2}^{(i)}C_{\theta_1,\theta_2}^{(j)}})\right)  \left(\alpha_{|i-j|-1}^{\delta/(2+\delta)}\right)\\
    &\qquad+ \sum_{i=1}^n\left(\frac{4}{n}+2n^{\delta/2}(C_{\theta_1,\theta_2}^{(i)}+D_{1,2}+\sqrt{C_{\theta_1,\theta_2}^{(i)}D_{1,2}})\right)\left(\alpha_{i-1}^{\delta/(2+\delta)}\right) + \Cov(Z_0,Z_0).
\end{align}
\end{proposition}
Note that this result holds for any parameterized Markov chain. In particular, when the Markov chain is stationary, $C_{\theta_1,\theta_2}^{(i)}=C_{\theta_1,\theta_2}^{(1)}\enspace\forall\enspace i$ and $\forall \theta \in \Theta$, and the expression~\cref{eq:var-ll} simplifies to
\begin{align}
    \nonumber
    \Var \left( r_n(\theta_1,\theta_2) \right) & < n\left(\frac{4}{n}+6n^{\delta/2}C_{\theta_1,\theta_2}^{(1)}\right)  \left(\sum_{k\geq 0}\alpha_{k}^{\delta/(2+\delta)}\right)\\
     &\qquad+ \left(\frac{4}{n}+2n^{\delta/2}(C_{\theta_1,\theta_2}^{(1)}+D_{1,2}+\sqrt{C_{\theta_1,\theta_2}^{(1)}D_{1,2}})\right)\left(\sum_{k\geq 1} \alpha_k^{\delta/(2+\delta)}\right) + \Cov(Z_0,Z_0).
\end{align}
If the sum $\sum_{k \geq 0} \alpha_{k}^{\delta/(2+\delta)}$ is infinite, the bound is trivially true. 
For it to be finite, of course, the coefficients $\alpha_k$ must decay to zero sufficiently quickly. 
For instance, \Cref{theorem:jones} shows that if the Markov chain is geometrically ergodic, then the $\alpha$-mixing coefficients are geometrically decreasing. We will use this fact when the Markov chain is non-stationary, as in Section~\ref{sec:nsedgm}. In the next section, however, we first consider the simpler stationary Markov chain setting where ergodic conditions are not explicitly imposed. We also note that unless only a finite number of $\alpha_k$ are nonzero, the sum $\sum_{k \geq 0} \alpha_{k}^{\delta/(2+\delta)}$ is infinite when $\delta=0$, and our results will typically require $\delta >0$.

\textcolor{purple}{}

\section{Stationary Markov Data-Generating Models}~\label{sec:sedgm}
 In this section we assume that corresponding to each transition kernel $p_{\theta},~\theta \in \Theta,$ there exists an invariant distribution $q^{(\infty)}_\theta \equiv q_{\theta}$ that satisfies
\[
    q_\theta(x) = \int p_\theta(x|y) q_\theta(dy) \quad \forall x \in \mathbb R^m, \theta \in \Theta.
\]

We will also use $q_\theta$ to designate the density of the invariant measure (as before, this is with respect to the Lebesgue or counting measure for continuous or discrete state spaces respectively). A Markov chain is stationary if its initial distribution is the invariant probability distribution, 
that is, $X_0\sim q_{\theta}$.

Observe that the PAC-Bayesian concentration bound in \Cref{thm:pac-bayes2} specifically requires bounding the mean and variance of the log-likelihood ratio $r_n(\theta,\theta_0)$. 
We ensure this by imposing regularity conditions on the log-ratio of the one-step transition kernels and the corresponding invariant densities. Specifically, we assume the following conditions that decouple the model parameters from the random samples, allowing us to verify the bounds in~\Cref{thm:pac-bayes2}.

\begin{assumption}~\label{assume:gen-lip}
 There exist positive functions $M^{(1)}_k(\cdot,\cdot)$ and $M^{(2)}_k(\cdot)$, $k \in \{1,2,\dots,m\}$ such that for any parameters $\theta_1,\theta_2 \in \Theta$, the log of the ratio of one-step transition kernels and the log of the ratio of the invariant distributions satisfy, respectively,
\begin{align}~\label{equation:lips_for_transition_probabilities}
       | \log p_{\theta_1}(x_1|x_0)-\log p_{\theta_2}(x_1|x_0) | \leq \sum_{k=1}^{m} M^{(1)}_k(x_1,x_0)|f^{(1)}_k(\theta_2,\theta_1)|~\forall~ (x_0,x_1), \text{and}
\end{align}
\begin{align}~\label{equation:lips_for_stationary_distribution}
    |\log q_{\theta_1}(x)-\log q_{\theta_2}(x)|\leq \sum_{k=1}^{m} M^{(2)}_k(x)|f^{(2)}_k(\theta_2,\theta_1)|\enspace\forall\enspace x.
\end{align}
We further assume that for some $\delta>0$, the functions $f^{(1)}_k, f^{(2)}_k$ and $M^{(1)}_k$ satisfy the following:
\begin{enumerate}[leftmargin=*,align=left]
    \item\label{assume:gen-lip(0)} there exist constants $C^{(t)}_k$ and measures $\rho_n \in \mathcal{F}$ such that $\int|f^{(t)}_k(\theta,\theta_0)|^{2+\delta}\rho_n(d\theta)<\frac{C_k^{(t)}}{n}$ for $t\in \{1,2\}$, $n\geq 1$ and $k\in \{1,2,\dots,m\}$, and
    \item there exists a constant $B$ such that $\int M^{(1)}_k(x_1,x_0)^{2+\delta}p_{\theta_j}(x_1|x_0)q_{\theta_j}^{(0)}(x_0)dx_1dx_0<B,\, k\in \{1,\dotsc,m\}$ and $j \in \{1,2\}$. 
\end{enumerate}
\end{assumption}

The following examples illustrate \cref{equation:lips_for_transition_probabilities} and \cref{equation:lips_for_stationary_distribution} for discrete and continuous state Markov chains. 
\begin{example}~\label{example:bd}
Suppose $\{\data\}$ is generated by the birth-death chain with parameterized transition probability mass function,
\begin{align*}
    p_{\theta}(j|i) = \begin{cases}
    \theta &~\text{if}~j=i-1,\\
    1-\theta&~\text{if}~j=i+1.
    \end{cases}
\end{align*}
In this example, the parameter $\theta$ denotes the probability of birth. We shall see that, $m=3$: $M^{(1)}_1(X_1,X_0)=I_{[X_1=X_0+1]}$, $M^{(1)}_2(X_1,X_0)=I_{[X_1=X_0-1]}$, and $M^{(1)}_3(X_1,X_0) = 1$. 
We also define $M^{(2)}_1(X_0)=1$, and set $M^{(2)}_2(X_0)$ and $M^{(2)}_3(X_0)$ both to $X_0-1$. Let $f^{(1)}_1(\theta,\theta_0) = \log\left[\frac{\theta_0}{\theta}\right]$,  $f^{(1)}_2(\theta,\theta_0) =\log\left[\frac{1-\theta_0}{1-\theta}\right]$, $f^{(1)}_3(\theta,\theta_0)=0$,
$f^{(2)}_1(\theta,\theta_0)=-f^{(2)}_3(\theta,\theta_0)=\log\left[\frac{1-\theta_0}{1-\theta}\right]$, and  $f^{(2)}_2(\theta,\theta_0)=\log\left[\frac{\theta_0}{\theta}\right]$. 
The derivation of these terms and that they satisfy the conditions of \Cref{assume:gen-lip} is provided in the proof of \Cref{thm:bd2}.
\end{example}

\begin{example}~\label{example:slm}
Suppose $\{\data\}$ is generated by the `simple linear' Markov model 
\begin{equation*}
    X_n=\theta X_{n-1}+W_n,
\end{equation*}
where $\{W_n\}$ is a sequence of i.i.d.\ standard Gaussian random variables. 
Then, $m = 2$, with 
$M^{(1)}_1(X_n,X_{n-1})=|X_nX_{n-1}|$, $M^{(1)}_2(X_n,X_{n-1})=X_n^2, M^{(2)}_1(x)=\frac{x^2}{2}$ and $M^{(2)}_2(X)=0$.
Corresponding to these, we have $f^{(1)}_1(\theta,\theta_0)=(\theta-\theta_0), f^{(1)}_2(\theta,\theta_0)=(\theta_0^2-\theta^2) ,f^{(2)}_1(\theta_0,\theta_0)=(\theta_0^2-\theta^2)$ and $f^{(2)}_2(\theta_0,\theta_0)=0$. The derivation of these quantities and that these satisfy the conditions of \Cref{assume:gen-lip} under appropriate choices of $\rho_n$ is shown in the proof of \Cref{ex:slm}.
\end{example}

Note that assuming the same number $m$ of $M^{(1)}_k$ and $M^{(2)}_k$ involves no  loss of generality, since these functions can be set to $0$. 
Both \cref{equation:lips_for_transition_probabilities} and \cref{equation:lips_for_stationary_distribution} can be viewed as generalized Lipschitz-smoothness conditions, recovering the usual Lipschitz-smoothness when $m = 1$ and when $f_k^{(t)}$ is Euclidean distance.
Our generalized conditions are useful for distributions like the Gaussian, where Lipschitz smoothness does not apply.
%
Observe also that by an application of Jensen's inequality,~\Cref{assume:gen-lip}~(\ref{assume:gen-lip(0)}) above implies that  for some constant $C>0$ 
and $k\in \{1,2,\dots,m\}, t \in \{1,2\}$, 
\begin{align}
\int |f^{(t)}_k(\theta,\theta_0)|\rho_n(d\theta) & \leq \frac{C}{n^{1/(2+\delta)}} < \frac{C}{\sqrt{n}}.
     \label{assume:gen-lip(2)}
\end{align}

\Cref{assume:gen-lip}~(\ref{assume:gen-lip(0)}) is satisfied in a variety of scenarios, for example, under mild assumptions on the partial derivatives of the functions $f^{(t)}_k$. 
To illustrate this, we present the following proposition.

\begin{proposition}~\label{prop:3}
Let $f(\theta,\theta_0)$ be a function on a bounded domain with bounded partial derivatives with $f(\theta_0,\theta_0)=0$. Let $\{\rho_n(\cdot)\}$ be a sequence of probability densities on $\theta$ such that $\E_{\rho_n}[\theta]=\theta_0$ and $\Var_{\rho_n}[\theta]=\frac{\sigma^2}{n}$ for some $\sigma>0$. Then, for some $C>0$,
\begin{equation}
    \int|f(\theta,\theta_0)|^{2+\delta}\rho_n(d\theta)<\frac{C}{n}.
\end{equation}
\end{proposition}

\begin{proof}
 Define $\partial_{\theta} f(\theta,\theta_0) := \frac{\partial f(\theta,\theta_0)}{\partial \theta}$ as the partial derivative of the function $f$. By the mean value theorem, $|f(\theta,\theta_0)|=|\theta-\theta_0||\partial_{\theta}  f(\theta^*,\theta_0)|$, for some $\theta^* \in [\min\{\theta,\theta_0\},\max\{\theta,\theta_0\}]$. Since the partial derivatives are bounded, there exists $L \in \mathbb R$ such that $\partial_{\theta} f(\theta^*,\theta_0)<L$, and $\int |f(\theta,\theta_0)|^{2+\delta}\rho_n(d\theta)<L^{2+\delta}\int |\theta-\theta_0|^{2+\delta} \rho_n(d\theta) $. 
Choose $G>0$ be such that $|\theta|<G$, then  $\left|\frac{\theta-\theta_0}{2G}\right|^{2+\delta}<\left|\frac{\theta-\theta_0}{2G}\right|^{2}$. Therefore, $\int |\theta-\theta_0|^{2+\delta} \rho_n(d\theta) < (2G)^{2+\delta} \Var\left[\frac{\theta}{2G}\right]<(2G)^{\delta}\frac{\sigma^2}{n}$. Now choosing $(2G)^{\delta}\sigma^2$ as $C$ completes the proof.
\end{proof}
If $\partial_{\theta}  f^{(t)}_k$ is continuous and $\Theta$ is compact, then $\partial_{\theta}  f^{(t)}_k$ is always bounded. Also observe that if $\E\left[M^{(1)}_k(X_1,X_{0})^{2+\delta}\right]<B$, without loss of generality we can use Jensen's inequality to conclude that, for all $0<a<2+\delta$,  $\E\left[M^{(1)}_k(X_1,X_{0})^a\right]<B^{\frac{a}{2+\delta}}<B$.

\sloppy We can now state the main theorem of this section. 
\begin{theorem}\label{thm:lip-gen}
  Let $\{X_0,\ldots,X_n\}$ be generated by a stationary, $\alpha$-mixing Markov chain parametrized by $\theta_0 \in \Theta$. Suppose that~\Cref{assume:gen-lip} holds and that the $\alpha$-mixing coefficients satisfy $\sum_{k \geq 1} \alpha_k^{\delta/(2+\delta)} < +\infty$. Furthermore, assume that $\mathcal{K}(\rho_n,\pi)\leq \sqrt{n}C$ for some constant $C>0$. Then, the conditions of~\Cref{thm:pac-bayes2} are satisfied with $\epsilon_n\in\mathrm{O}\left(\max(\frac{1}{\sqrt{n}},\frac{n^{\delta/2}}{n})\right)$.
\end{theorem}
Theorem~\ref{thm:lip-gen} is satisfied by a large class of Markov chains, including chains with countable and continuous state spaces. In particular, if the Markov chain is geometrically ergodic, then it follows from \cref{eq:alpha_mixing-coeff-sum} that $\sum_{k \geq 1} \alpha_k^{\delta/(2+\delta)} < +\infty$.
Observe that in order to achieve $O(\frac{1}{\sqrt{n}})$ convergence, we need $\delta \le 1$. 
Note also that as $\delta$ decreases, satisfying the condition $\sum_{k \geq 1} \alpha_k^{\delta/(2+\delta)}$ requires the Markov chain to be faster mixing.

We now illustrate Theorem~\ref{thm:lip-gen} for a number of Markov chain models. First, consider a birth-death Markov chain on a finite state space. 
\begin{proposition}~\label{thm:bd}
Suppose the data-generating process is a birth-death Markov chain, with one-step transition kernel parametrized by the birth probability $\theta_0 \in \Theta$. Let $\mathcal{F}$ be the set of all Beta distributions. We choose the prior to be a Beta distribution with parameters $\alpha$ and $\beta$. Then, the conditions of Theorem~\ref{thm:lip-gen} are satisfied and $\epsilon_n\in\mathrm{O}\left(\frac{1}{\sqrt{n}}\right)$.
\end{proposition}
\begin{proof}
The proof of \Cref{thm:bd} follows from the more general \Cref{thm:bd-ns}, by fixing the initial distribution to the invariant distribution under $\theta_0$.
\end{proof}
%
The birth death chain on the finite state space is, of course, geometrically ergodic and the alpha mixing coefficients $\alpha_k$ decay geometrically. Note that the invariant distribution of this Markov chain is uniform over the state space, and consequently this is a particularly simple example. A more complicated and more realistic example is a birth-death Markov chain on the positive integers. We note that if the probability of birth in a birth-death Markov chain on positive integers is greater than $0.5$, then the Markov chain is transient, and consequently, not ergodic. As seen in \Cref{example:bd}, we denoted the probability of birth as $\theta$. Hence, our prior should be chosen in such a fashion that the support should be within $(0,0.5)$. For that purpose, we define the class of scaled beta distribution.
\begin{defn}[{Scaled Beta}]~\label{def:scaled-beta}
If $X$ is a beta distribution on with parameters $\alpha$ and $\beta$, then $Y$ is said to be a scaled beta distribution with same parameters on the interval $(c,m+c)$ if,
\begin{align*}
    Y & = mx+c\enspace ; \enspace (m,c)\in \mathbb{R}^2
\end{align*}
and in that case, via transformation of variables, the pdf of $Y$ is obtained as,
\begin{align*}
    f(y) = \begin{cases}
    \frac{m}{\mathrm{Beta}(\alpha,\beta)}\left(\frac{y-c}{m}\right)^{\alpha-1}\left(1-\frac{y-c}{m}\right)^{\beta-1} &  \mathrm{if} y\in (c,m+c)\\
    \quad 0 & \mathrm{otherwise}.
    \end{cases}
\end{align*}
\end{defn}
It follows that under such circumstances, $\E[Y]=m\frac{\alpha}{\alpha+\beta}$ and $\Var[Y]=m^2\frac{\alpha\beta}{(\alpha+\beta)^2(\alpha+\beta+1)}$.
The scaled beta distribution is a popular distribution which finds use in a variety of practical applications. One example of such scaled beta distribution might be when $m=0.5$ and $c=0$. Then it becomes a beta distribution rescaled to have support on $(0,\frac{1}{2})$. Another example is when $m=2$ and when $c=-1$. Then it becomes a beta distribution rescaled to have support on $(-1,1)$.
\begin{proposition}~\label{thm:bd2}
Suppose the data-generating process is a positive recurrent birth-death Markov chain on the positive integers parameterized by the birth probability $\theta_0\in(0,\frac{1}{2})$.  Further let $\mathcal{F}$ be the set of all Beta distributions rescaled to have support $(0,\frac{1}{2})$. We choose the prior to be a scaled Beta distribution on $(0,1/2)$ with parameters $\alpha$ and $\beta$. Then, the conditions of Theorem~\ref{thm:lip-gen} are satisfied with $\epsilon_n\in\mathrm{O}\left(\frac{1}{\sqrt{n}}\right)$.
\end{proposition}
\begin{proof}
The proof of \Cref{thm:bd2} follows from that of \Cref{thm:bd2-ns} by fixing the initial distribution to the invariant distribution under $\theta_0$.
\end{proof}
Note that if the transition probability of jumping from state $i$ to state $i+1$ (i.e., the `birth' probability) is greater than $\frac{1}{2}$, then the birth-death chain is transient. Therefore, we restrict to only those cases when the probability of birth is less than $\frac{1}{2}$. 
Unlike with the finite state-space, 
the invariant distribution now depends on the parameter $\theta \in \Theta$, and verification of the conditions of the proposition is more involved. 

Both Proposition~\ref{thm:bd} and Proposition~\ref{thm:bd2} assume a discrete state space. The next example considers a strictly stationary {simple linear model} (as defined in~\Cref{example:slm}), which has a continuous, unbounded state space.
\begin{proposition}~\label{example:slm-stationary}
Suppose the data-generating model is a strictly stationary {simple linear model} satisfying the equation
\begin{align}
    X_n & = \theta_0 X_{n-1}+W_n,
\end{align}
where $\{W_n\}$ are i.i.d. standard Gaussian random variables and $|\theta_0| < 1$. Let $\mathcal F$ be the space of all parameterized distributions with support $(-1,1)$ and suppose that $\mathcal{F}$ is the class of all beta distributions rescaled to have the support $(-1,1)$. Then, the conditions of \Cref{thm:lip-gen} are satisfied with $\epsilon_n \in \mathrm{O}\left(\frac{1}{\sqrt{n}}\right)$.
\end{proposition}
\begin{proof}
The proof that simple linear model satisfies \Cref{assume:gen-lip} is deferred to the proof of \Cref{ex:slm}. The simple linear model with $|\theta_0| < 1$ has geometrically decreasing (and therefore summable) $\alpha$-mixing coefficients as a consequence of~\cite[eq. (15.49)]{meyntweedie} and~\Cref{theorem:jones}. Combining these two facts, it follows that the conditions of \Cref{thm:lip-gen} are satisfied. 
\end{proof}
Observe that Theorem~\ref{thm:pac-bayes} (and Corollary~\ref{thm:pac-bayes2}) are general, and hold for {\it any} dependent data-generating process. Therefore, there can be Markov chains that satisfy these, but do not satisfy Assumption~\ref{assume:gen-lip} which entails some loss of generality. However, as our examples demonstrate, common Markov chain models do indeed satisfy the latter assumption.

\section{Non-Stationary, Ergodic Markov Data-Generating Models}~\label{sec:nsedgm}
We call a time-homogeneous Markov chain \textit{non-stationary} if the initial distribution $q^{(0)}$ is not the invariant distribution. 
There are two sets of results in this setting: in~\Cref{thm:bound-func} and~\Cref{thm:gen-nonstationary-hajek} we explicitly impose the $\alpha$-mixing condition, while in~\Cref{thm:lip-gen-ns} we impose a $V$-geometric ergodicity condition (\Cref{define:v-ergodicity}). As seen in \cref{eq:alpha_mixing-coeff-sum} if the Markov chain is also geometrically ergodic, then $\forall \enspace \delta>0$, $\sum \alpha_k^{\delta/(2+\delta)}<\infty$. This condition can be relaxed, albeit at the risk of more complicated calculations that, nonetheless, mirror those in the geometrically ergodic setting. A common thread through these results is that we must impose some integrability or regularity conditions on the functions $M^{(1)}_k$.

First, in~\Cref{thm:bound-func} we assume that the $M^{(1)}_k$ functions in~\Cref{assume:gen-lip} are uniformly bounded and that the $\alpha$-mixing condition is satisfied. This result holds for both discrete and continuous state space settings. 

\begin{theorem}\label{thm:bound-func}
  Let $\{\data\}$ be generated by an $\alpha$-mixing Markov chain parametrized by $\theta_0 \in \Theta$ with transition probabilities satisfying \Cref{assume:gen-lip} and with known initial distribution $q^{(0)}$. Let $\{\alpha_k\}$ be the $\alpha$-mixing coefficients under $\theta_0$, and assume that $\sum_{k\geq 1} \alpha_k^{{\delta}/{(2+\delta)}} < +\infty$. Suppose that there exists $B \in \mathbb R$ such that $\sup_{x,y}|M^{(1)}_k(x,y)| < B$ for all $k \in \{1,2,\dots,m\}$ in~\Cref{assume:gen-lip}. Furthermore, assume that there exists $\rho_n \in \mathcal{F}$ such that $\mathcal{K}(\rho_n,\pi)\leq \sqrt{n}C$ for some constant $C>0$. If the initial distribution $q^{(0)}$ satisfies $E_{q^{(0)}}|M^{(2)}_k(X_0)|^2 < +\infty$ for all $k \in \{1,2,\dots,m\}$, then the conditions of \Cref{thm:pac-bayes2} are satisfied with $\epsilon_n=\mathrm{O}\left(\max(\frac{1}{\sqrt{n}},\frac{n^{\delta/2}}{n})\right)$.
\end{theorem}
The following result in~\Cref{thm:bd-ns} illustrates Theorem~\ref{thm:bound-func} in the setting of a finite state birth-death Markov chain.
\begin{proposition}~\label{thm:bd-ns}
Suppose the data-generating process is a finite state birth-death Markov chain, with one-step transition kernel parametrized by the birth probability $\theta_0$. Let $\mathcal{F}$ be the set of all Beta distributions. We choose the prior to be a Beta distribution with parameters $\alpha$ and $\beta$. Then, the conditions of \Cref{thm:bound-func} are satisfied with $\epsilon_n=\mathrm{O}\left(\frac{1}{\sqrt{n}}\right)$ for any initial distribution $q^{(0)}$.
\end{proposition}

 Theorem~\ref{thm:bound-func} also applies to data generated by Markov chains with countably infinite state spaces, so long as the class of data-generating Markov chains is strongly ergodic and the initial distribution has finite second moments. The following example demonstrates this in the setting of a birth-death Markov chain on the positive integers, where the initial distribution is assumed to have finite second moments.
 
\begin{proposition}~\label{thm:bd2-ns}
Suppose the data-generating process is a birth-death Markov chain on the non-negative integers, parameterized by the probability of birth $\theta_0\in(0,\frac{1}{2})$. Further let $\mathcal{F}$ be the set of all Beta distributions rescaled upon the support $(0,\frac{1}{2})$. Let $q^{(0)}$ be a probability mass function on non-negative integers such that $\sum_{i=1}^{\infty} i^2 q^{(0)}(i) < +\infty$. We choose the prior to be a scaled Beta distribution on $(0,1/2)$ with parameters $\alpha$ and $\beta$. Then, the conditions of \Cref{thm:bound-func} are satisfied with $\epsilon_n=\mathrm{O}\left(\frac{1}{\sqrt{n}}\right)$.
\end{proposition}

Since continuous functions on a compact domain are bounded, we have the following (easy) corollary (stated without proof).

\begin{corollary}\label{cor:bound-ss}
  Let $\{\data\}$ be generated by an $\alpha$-mixing Markov chain parametrized by $\theta_0 \in \Theta$ on a compact state space, and with initial distribution $q^{(0)}$. Suppose the $\alpha$-mixing coefficients satisfy $\sum_{k\geq 1} \alpha_k^{{\delta}/{(2+\delta)}} < +\infty$, and that~\Cref{assume:gen-lip} holds with continuous functions $M^{(1)}_k(\cdot,\cdot)$, $k \in \{1,2,\dots,m\}$.  Furthermore, assume that {there exists $\rho_n$ such that} $\mathcal{K}(\rho_n,\pi)\leq \sqrt{n}C$ for some constant C. Then, Theorem~\ref{thm:bound-func} is satisfied with $\epsilon_n=\mathrm{O}\left(\max(\frac{1}{\sqrt{n}},\frac{n^{\delta/2}}{n})\right)$.
\end{corollary}

In general the $M^{(1)}_k$ functions will not be uniformly bounded (consider the case of a Gaussian-Markov simple linear model in~\Cref{example:slm}), and stronger conditions must be imposed on the data-generating Markov chain itself. The following assumption imposes a `drift' condition from~\cite{hajek}. Specifically,~\cite[Theorem 2.3]{hajek} shows that under the conditions of \Cref{assume:hajek}, the moment generating function of an aperiodic Markov chain $\{X_n\}$ can be upper bounded by a function of the moment generating function of $X_0$. Together with the $\alpha$-mixing condition,~\Cref{assume:hajek} implies that this Markov data generating process satisfies~\Cref{thm:pac-bayes2}.

\begin{assumption}\label{assume:hajek}
Consider a Markov chain $\{X_n\}$ parameterized by $\theta_0 \in \Theta$.  Let $\mathcal{M}_{-\infty}^n$ denote the $\sigma$-field generated by $\{X_{-\infty},\dots,X_{n-1},X_n\}$. Denote the stochastic process $\{M^k_{n}\}:=\{M^{(1)}_k(X_n,X_{n-1})\}$;  recall $M^{(1)}_k$, for each $k=1,\ldots,m_1$, are defined in \Cref{assume:gen-lip}. For each $k =1,\ldots,m$, the process $\{M^k_{n}\}$ is assumed to satisfy the following conditions:
\begin{itemize}
    \item The drift condition holds for $\{M^k_{n}\}$, i.e. $\E\left[M^k_{n}-M^k_{n-1}|\mathcal{M}_{-\infty}^{n-1}, M^k_{n-1}>a\right]\leq -\epsilon$ for some $\epsilon,a>0$.
    \item For some $\lambda>0$ and $\mathcal{D}>0$, $\E\left[e^{\lambda(M^k_{n}-M^k_{n-1})}|\mathcal{M}_{-\infty}^{n-1}\right]\leq \mathcal {D}$. 
\end{itemize}
\end{assumption}

Under this drift condition, the next theorem shows that~\Cref{thm:pac-bayes2} is satisfied. 

\begin{theorem} \label{thm:gen-nonstationary-hajek}
  Let $\{X_0,\ldots,X_n\}$ be generated by an aperiodic $\alpha$-mixing Markov chain parametrized by $\theta_0 \in \Theta$ and initial distribution $q^{(0)}$. Suppose that~\Cref{assume:gen-lip} and \Cref{assume:hajek} hold, and that the $\alpha$-mixing coefficients satisfy $\sum_{k \geq 1} \alpha_k^{\delta/(2+\delta)} < +\infty$. Furthermore, assume $\mathcal{K}(\rho_n,\pi)\leq \sqrt{n}C$ for some constant $C>0$. If $\int e^{\lambda M^{(1)}_k(y,x)}p_{\theta_0}(y|x)q_1^{(0)}(x) dx < +\infty$ for all $k=1,\ldots,m_1$, then the conditions of \Cref{thm:pac-bayes2} are satisfied with $\epsilon_n=\mathrm{O} \left(\max(\frac{1}{\sqrt{n}},\frac{n^{\delta/2}}{n})\right)$.
\end{theorem}

Verifying the conditions in Theorem~\ref{thm:gen-nonstationary-hajek} can be quite challenging in general.
Instead, we suggest a different approach that requires $V$-geometric ergodicity.
Unlike the drift condition in~\Cref{assume:hajek}, $V$-geometric ergodicity additionally requires the existence of a petite set. However, geometric ergodicity is a fairly standard property that has already been established from a number of important Markov chains.
As noted before, geometric ergodicity implies $\alpha$-mixing of the Markov chain with geometrically decaying mixing coefficients. As with \Cref{thm:gen-nonstationary-hajek}, we assume that the Markov chain is aperiodic. It is well known that in case of a periodic Markov chain with period $d$, the notion of geometric ergodicity only makes sense for the $d$-chain $\{X_{dn}\}.$
\begin{theorem}\label{thm:lip-gen-ns}
 Let $\{\data\}$ be generated by an aperiodic Markov chain parametrized by $\theta_0 \in \Theta$ with known initial distribution $q^{(0)}$, and assumed to be $V$-geometrically ergodic for some $V : \mathbb R^m \to [1,\infty)$. Suppose that ~\Cref{assume:gen-lip} holds and $\int M^{(1)}_k(y,x)^{2+\delta} p_{\theta_0}(y|x) dy  < V(x)\enspace \forall \enspace k,x \text{ and some }\delta>0$. Furthermore, assume that $\mathcal{K}(\rho_n,\pi)\leq \sqrt{n}C$ for some constant $C>0$. If the initial distribution $q^{(0)}$ satisfies $E_{q^{(0)}}[V(X_0)] < +\infty$, then the conditions of \Cref{thm:pac-bayes2} are satisfied with $\epsilon_n=\mathrm{O}\left(\max(\frac{1}{\sqrt{n}},\frac{n^{\delta/2}}{n})\right)$.
\end{theorem}
{
%
%
}
The following Proposition~\ref{ex:slm} shows, the simple linear model satisfies Theorem~\ref{thm:lip-gen-ns} when the parameter $\theta_0$ is suitably restricted.  

\begin{proposition}\label{ex:slm}
Consider the {simple linear model} satisfying the equation
\begin{align}
    X_n & = \theta_0 X_{n-1}+W_n,
\end{align}
where $\{W_n\}$ are i.i.d. standard Gaussian random variables and $|\theta_0| < 2^{\frac{1}{4+2\delta}-1}$ for $\delta > 0$. Let $\mathcal{F}$ be the space of all scaled Beta distributions on $(-1,1)$ and suppose the prior $\pi$ is a uniform distribution on $(-1,1)$. Then, the conditions of \Cref{thm:lip-gen-ns} are satisfied with $\epsilon_n \in \mathrm{O}\left(\max(\frac{1}{\sqrt{n}},\frac{n^{\delta/2}}{n})\right)$, if the initial distribution $q^{(0)}$ satisfies $ E_{q^{(0)}}[X_0^{4+2\delta}] < +\infty$.
\end{proposition}

\section{Misspecified Models}
{
We shown next how our results can be extended to the misspecified model setting. 
Assume that the true data generating distribution is parametrized by $\theta_0 \not \in \Theta$. Let $\theta^*_n := \arg \min_{\theta\in \Theta} \kld$ represent the closest parametrized distribution in the variational family to the data-generating distribution. Further, assume that our usual hypotheses are satisfied:
\begin{enumerate}[align=left,leftmargin=*]
    \item $ \int \E[r_n(\theta,\theta^*_n)] \rho_n(d\theta) \leq n\epsilon_n $,
    \item $\int \Var \left( r_n(\theta,\theta^*_n
    ) \right)\rho_n (d\theta)\leq n\epsilon_n$.
\end{enumerate}

Now, since $ r_n(\theta,\theta_0) =  r_n (\theta,\theta^*_n)+r_n(\theta^*_n,\theta_0)$, we have
\begin{equation}~\label{eq:miss-kl}
    \int \kld \rho_n(d\theta)\leq \E[r_n(\theta_0,\theta^*_n)]+n\epsilon_n.
\end{equation}
Similarly, decomposing the variance it follows that
\begin{align}
    \Var[r_n(\theta,\theta_0)] & = \Var[r_n (\theta,\theta^*_n)]+\Var[r_n(\theta^*_n,\theta_0)]+2\Cov[r_n (\theta,\theta^*_n),r_n(\theta^*_n,\theta_0)].\\
    \intertext{Using the fact that $2ab\leq{a^2+b^2}$ to the covariance term $2\Cov[r_n (\theta,\theta^*_n),r_n(\theta^*_n,\theta_0)]=2\E \left[\left(r_n (\theta,\theta^*_n)-\E[r_n (\theta,\theta^*_n)]\right)\left(r_n(\theta^*_n,\theta_0)-\E[r_n(\theta^*_n,\theta_0)]\right)\right]$, we have }
    \Var[r_n(\theta,\theta_0)] & \leq 2\Var[r_n (\theta,\theta^*_n)]+2\Var[r_n(\theta^*_n,\theta_0)].
    \intertext{Integrating both sides with respect to $\rho_n(d\theta)$ we get}
    \int \Var[r_n(\theta,\theta_0)] \rho_n(d\theta) &  \leq 2\int \Var[r_n (\theta,\theta^*_n)] \rho_n(d\theta)+2\int \Var[r_n(\theta^*_n,\theta_0)] \rho_n(d\theta)\nonumber\\
    & \leq 2n\epsilon_n +2\Var[r_n(\theta^*_n,\theta_0)].~\label{eq:miss-var}
\end{align}
Consequently, we arrive at the following result:
\begin{theorem}\label{thm:pac-bayes-miss}
Let $\mathcal{F}$ be a subset of all probability distributions parameterized by $\Theta$. Let  $\theta^*_n = \arg \min_{\theta\in \Theta} \kld$ and assume there exist $\epsilon_n >0$ and $\rho_n\in \mathcal{F}$ such that 
\begin{enumerate}[align=left,leftmargin=*]
    \item $ \int \E[r_n(\theta,\theta^*_n)] \rho_n(d\theta) \leq n\epsilon_n $,
    \item $\int \Var \left( r_n(\theta,\theta^*_n
    ) \right)\rho_n (d\theta)\leq n\epsilon_n$, and
    \item $\mathcal{K}(\rho_n,\pi)\leq n\epsilon_n$.
\end{enumerate}
Then, for any $\alpha \in (0,1)$ and $(\epsilon, \eta) \in (0,1)\times(0,1)$,
\begin{align}~\label{eq:thmmiss}
    P\left[ \int  D_{\alpha}(P^{(n)}_{\theta},P^{(n)}_{\theta_0})\Tilde{\pi}_{n,\alpha}(d\theta|X^{(n)} )\leq \frac{ (\alpha+1)n\epsilon_n+\E[r_n(\theta_0,\theta^*_n)]+\alpha\sqrt{\frac{2n\epsilon_n+2\Var[r_n(\theta^*_n,\theta_0)]}{\eta}}-\log(\epsilon) } {  1-\alpha}  \right]\\
    \nonumber \geq 1-\epsilon-\eta.
\end{align}
\end{theorem}
}

The proof of this theorem is straightforward and follows from the proof of \Cref{thm:pac-bayes} by plugging in the upper bounds for KL-divergence from \cref{eq:miss-kl}, and variance from \cref{eq:miss-var}  to \Cref{eq:proof-step-miss}. A sketch of the proof is presented in the appendix.

\section{Conclusion}

Concentration of the KL-VB model risk, in terms of the expected \alpren divergence, is well established under the i.i.d. data generating model assumption.
Here, we extended this to the setting of Markov data generating models, linking the concentration rate to the mixing and ergodic properties of the Markov model. 
Our results apply to both stationary and non-stationary Markov chains, as well as to the situation with misspecified models.
There remain a number of open questions. An immediate one is to extend the current analysis to continuous-time Markov chains and Markov jump processes, possibly using uniformization of the continuous time model. Another direction is to extend this to the setting of non-homogeneous Markov chains, where analogues of notions such as stationarity are less straightforward. Further, as noted in the introduction,~\cite{pati} establish PAC-Bayes bounds under slightly weaker `existence of test functions' conditions, while our results are established under the stronger conditions used by~\cite{pierre} for the i.i.d. setting. Weakening the conditions in our analysis is important, but complicated. 
A possible path is to build on results from~\cite{birge} who provides conditions form the existence of exponentially powerful test functions exist for distinguishing between two Markov chains. It is also known that there exists a likelihood ratio test separating any two ergodic measures~\cite{ryabko}. However, leveraging these to establish the PAC-Bayes bounds for the KL-VB posterior is a challenging effort that we leave to future papers. 
Finally it is of interest to generalize our PAC-bounds to posterior approximations beyond KL-variational inference, such as $\alpha$-R\'enyi posterior approximations~\cite{li2016renyi}, and loss-calibrated posterior approximations~\cite{LaSiGh2011, jaiswalstat2020}.

\appendix
\section{Definitions Related to Markov Chains}
As noted before, ergodicity plays an acute role in establishing our results. We consolidate various definitions used throughout the paper in this appendix. We assume that the parameterized Markov chains possess an invariant probability density or mass function $q_\theta$ under parameter $\theta \in \Theta$. 






We defined $V$-geometric ergodicity in the previous sections. In this section, we provide a sufficient condition for a Markov chain to be $V$-geometrically ergodic. First, we recall the definition of petite sets,


\begin{defn}[{Petite Sets}] Let $\data$ be $n$ samples from a Markov chain taking values on the state space $\mathcal{X}$. Let $C$ be a set. We shall call $C$ to be $v_q$ petite if 
\begin{align*}
    K_q(x,B)\geq \upsilon_q(B)
\end{align*}
for all $x\in C$ and $B\in \mathcal{B(X)}$, and a non-trivial measure $\upsilon_q$ on $\mathcal{B(X)}$ 
\end{defn}

Now, let $\Delta V(x):=E[V(X_n)|X_{n-1}=x]-V(x)$ for $V : S \to [1,\infty)$.

\begin{assumption}[{Drift condition}]~\label{assume:drift} Suppose the chain $\{X_n\}$ is, aperiodic and \psirr. Let there exists a petite set $C$, constants $b<\infty, \beta>0$, and a non-trivial function $V : S \to [1,\infty)$ satisfying
\begin{align}
    \Delta V(x)\leq -\beta V(x) + bI_{x\in C}\enspace \forall x \in S.
\end{align}
\end{assumption}
If a Markov chain drifts towards a petite set then it is $V$-geometrically ergodic. Suppose, for simplicity, that $V(x)=|X|$. Then, the drift condition becomes $E[|X_n\|X_{x-1}]-|X_{n-1}|=-\beta |X_n|+bI_{X_n\in C}$. The left hand side of this equation represents the change in the state of the Markov chain in one time epoch. Thus, the condition in \Cref{assume:drift} essentially states that the Markov chain drifts towards a petite set $C$ and then, once it reaches that set, moves to any point in the state space with at least some probability independent of $C$. 

\begin{theorem}[{Geometrically ergodic theorem}]~\label{theorem:v-geom}
  Suppose that $\{X_n\}$ is satisfies \Cref{assume:drift}. Then, the set $S_V=\{x:V(x)<\infty\}$ is absorbing, i.e. $P_\theta(X_1\in S_V|X_0=x)=1$ $\forall x\in S_V$, and full, i.e. $\psi(S_V^c)=0$. Also, $\exists\enspace \text{constants }r>1, R<\infty$ such that, for any $A \in \mathcal{B}(S)$,
  \begin{equation}
      \left\|P_\theta(X_n \in A | X_0=x)- \int_A q_\theta(y) dy \right\|_V \leq R r^{-n}V(x).
  \end{equation}
\end{theorem}

Any aperiodic and \psirr Markov chain satisfying the drift condition is geometrically ergodic. A consequence of \Cref{equation:f-norm-monotone} is that if, $\{X_n\}$ is $V$-geometrically ergodic, then for any other function $U, \text{ such that }|U|<V$, it is also $U$-geometrically ergodic. In essence, a geometrically ergodic Markov chain is asymptotically uncorrelated in a precise sense. Recall $\rho$-mixing coefficients defined as follows. Let $\mathcal{A}$ be a sigma field and $\mathcal{L}^2(A)$ be the set of square integrable, real valued, $\mathcal{A}$ measurable functions.

\begin{defn}[{$\rho$-mixing coefficient}]~\label{def:rho-mix}
Let $\mathcal{M}_i^j$ denote the sigma field generated by the measures $X_k, \text{ where } i\leq k\leq j$. Then,
\begin{equation}
    \rho_k = \underset{t>0}{\sup}\underset{(f,g) \in \mathcal{L}^2\left(\mathcal{M}_{-\infty}^t\right) \times  \mathcal{L}^2\left(\mathcal{M}_{t+k}^{\infty}\right)}{\sup}\left|\mathrm{Corr}(f,g)\right|,
\end{equation}
where $\mathrm{Corr}$ is the correlation function.
\end{defn}

\begin{theorem}\label{theorem:jones}
  If $X_n$ is geometrically ergodic, then it is $\alpha$-mixing. That is, there exists a constant $c>0$ such that $\alpha_k \in O(e^{-ck})$.
\end{theorem}
\begin{proof}
By \cite[Theorem 2]{jones} it follows that a geometrically ergodic Markov chain is asymptotically uncorrelated with $\rho$-mixing coefficients (see~\cref{def:rho-mix}) that satisfy $\rho_k \in O(e^{-ck})$. Furthermore, it is well known that~\cite{bradley,jones} $\alpha_k\leq \frac{1}{4}\rho_k$, implying $\alpha_k \in O(e^{-ck})$.
\end{proof}

\section{Bounding the KL-divergence between Beta distributions}
The following results will be utilized in the proofs of Propositions~\ref{thm:bd-ns},~\ref{thm:bd2-ns} and~\ref{ex:slm}.
\begin{lemma}~\label{prop:kl-prior-unif}
Let $\theta_0\in(0,1)$. Let, $\rho_n$ be a sequence of Beta distributions with parameters $\alpha_n=n\theta_0$ and $\beta_n=n(1-\theta_0)$. Let $\pi$ denote an uniform distribution, $U(0,1)$. Then, $\mathcal{K}(\rho_n,\pi)<C+\frac{1}{2}\log (n)$, for some constant $C>0$. 
\end{lemma}

\begin{proof}
The KL divergence $\mathcal{K}(\rho_n,\pi)$ can be written as $\int \log \left(\frac{\rho_n}{\pi}\right) \rho_n(d\theta)$. Since $\pi$ is uniform, $\pi(\theta)=1$ whenever $\theta\in (0,1)$. Hence, the KL-divergence can be written as the negative of the entropy of $\rho_n$
$
 \int_0^1 \log \left(\rho_n\left(\theta\right)\right) \rho_n(d\theta),
$
which can be written as 
\begin{align}
\mathcal{K}(\rho_n,\pi) = (\alpha_n-1)\psi(\alpha_n)+(\beta_n-1)\psi(\beta_n)-(\alpha_n+\beta_n-2)\psi(\alpha_n+\beta_n)-\log \text{Beta}(\alpha_n,\beta_n),
\end{align}
where $\psi$ is the digamma function. Using Stirling's approximation on $\text{Beta}(\alpha_n,\beta_n)$ yields,
\begin{align*}
    \text{Beta}(\alpha_n,\beta_n) = \sqrt{2\pi} \frac{\alpha_n^{\alpha_n-1/2}\beta_n^{\beta_n-1/2}}{(\alpha_n+\beta_n)^{\alpha_n+\beta_n-1/2}}(1+o(1)).
\end{align*}
Plugging in the values of $\alpha_n$ and $\beta_n$, we get,
\begin{align*}
     \text{Beta}(\alpha_n,\beta_n) & = \sqrt{2\pi} \frac{(n\theta_0)^{(n\theta_0)-1/2}(n(1-\theta_0))^{(n(1-\theta_0))-1/2}}{(n\theta_0+n(1-\theta_0))^{n\theta_0+n(1-\theta_0)-1/2}}\left(1+o(1)\right)\\
      & = \sqrt{2\pi} \frac{n^{n-1}\theta_0^{n\theta_0-1/2}(1-\theta_0)^{n(1-\theta_0)-1/2}}{n^{n-1/2}}(1+o(1))\\
      & = \sqrt{\frac{2\pi}{n}}\theta_0^{n\theta_0-1/2}(1-\theta_0)^{n(1-\theta_0)-1/2} (1+o(1)).
\end{align*}
Setting $C_1=\log(\sqrt{2\pi})$, we have,
\begin{align*}
    -\log  \text{Beta}(\alpha_n,\beta_n) & <\log(1+o(1))+ C_1+\frac{1}{2}\log (n)-(n\theta_0-1/2)\log(\theta_0)-(n(1-\theta_0)-1/2)\log (1-\theta_0).
\end{align*}
Now, we analyze the term $(\alpha_n-1)\psi(\alpha_n)$. From~\cite{alzer1997some} we have that $\log(x)-\frac{1}{x}<\psi(x)<\log(x)-\frac{1}{2x}\  \forall \enspace x>0$. Without loss of generality, assume $\alpha_n>1$, giving
\begin{align*}
    (\alpha_n-1)\psi(\alpha_n) & <(\alpha_n-1)\log(\alpha_n)-\frac{\alpha_n-1}{2\alpha_n}.
    \intertext{If not, then we would have obtained using the lower bound of $\psi(x)$,}
    (\alpha_n-1)\psi(\alpha_n) & <(\alpha_n-1)\log(\alpha_n)-\frac{\alpha_n-1}{\alpha_n},
    \intertext{and we could have proceeded similarly. By plugging in the value of $\alpha_n$, we get}
     (\alpha_n-1)\psi(\alpha_n) & < (n\theta_0-1)\log(n\theta_0)-\left(1/2-\frac{1}{2n\theta_0}\right)\\
     & = (n\theta_0-1)\log(n)+(n\theta_0-1)\log(\theta_0) -\left(1/2-\frac{1}{2n\theta_0}\right).
\end{align*}
Similarly assuming $\beta_n>1$, we get the following upper bound for $(\beta_n-1)\psi(\beta_n)$, giving
\begin{align*}
    (\beta_n-1)\psi(\beta_n) < (n(1-\theta_0)-1)\log(n)+(n(1-\theta_0)-1)\log(1-\theta_0)-\left(1/2-\frac{1}{2n(1-\theta_0)}\right)
\end{align*}
Combining all the terms, we get, for $C^{(n)}_2=C_1-(1/2-\frac{1}{2n\theta_0})-(1/2-\frac{1}{2n(1-\theta_0)})$
\begin{align*}
    (\alpha_n-1)\psi(\alpha_n) & +(\beta_n-1)\psi(  \beta_n)-(\alpha_n+\beta_n-2)\psi(\alpha_n+\beta_n)-\log \text{Beta}(\alpha_n,\beta_n)  \\
    & < C^{(n)}_2 - \frac{1}{2}(\log(\theta_0)+\log(1-\theta_0))+(n\theta_0-1)\log(n)+(n(1-\theta_0)-1)\log(n)\\
    & \qquad +\frac{1}{2}\log (n)-(\alpha_n+\beta_n-2)\psi(\alpha_n+\beta_n)+\log(1+o(1)).\\
    \intertext{ Now plugging in the values of $\alpha_n$ and $\beta_n$ in $(\alpha_n+\beta_n-1)\psi(\alpha_n+\beta_n)$, we get,}
    (\alpha_n-1)\psi(\alpha_n) & +(\beta_n-1)\psi(  \beta_n)-(\alpha_n+\beta_n-2)\psi(\alpha_n+\beta_n)-\log \text{Beta}(\alpha_n,\beta_n)  \\
    & < C^{(n)}_2-\frac{1}{2}(\log(\theta_0)-\log(1-\theta_0))+(n-2)\log(n)+\frac{1}{2}\log (n)-(n-2)\psi(n)\\&\qquad\qquad\qquad\qquad\qquad\qquad\qquad\qquad\qquad\qquad\qquad +\log(1+o(1)).
\end{align*}
By using the lower bound of $\psi(x)$, we get that, $-\psi(x)<-\log(x)+\frac{1}{x}$. Plugging this into the above equation, we get,
\begin{align*}
    (\alpha_n-1)\psi(\alpha_n) & +(\beta_n-1)\psi(  \beta_n)-(\alpha_n+\beta_n-2)\psi(\alpha_n+\beta_n)-\log \text{Beta}(\alpha_n,\beta_n)  \\
     & < C^{(n)}_2-\frac{1}{2}(\log(\theta_0)-\log(1-\theta_0))+(n-2)\log(n) - (n-2)\log(n)+\frac{1}{2}\log(n)\\ &\qquad\qquad\qquad\qquad\qquad\qquad\qquad\qquad\qquad\qquad\qquad\qquad +\frac{n-2}{n} +\log(1+o(1))\\
     & = C^{(n)}_2-\frac{1}{2}(\log(\theta_0)-\log(1-\theta_0))+\frac{1}{2}\log(n)+\frac{n-2}{n}+\log(1+o(1)).\\
     \intertext{Since $\frac{1}{2n\theta_0}+\frac{1}{2n(1-\theta_0)})<2$, $C^{(2)}_n$ can be upper bounded by $C_1+1$. Using this fact, we get}
    (\alpha_n-1)\psi(\alpha_n) & +(\beta_n-1)\psi(  \beta_n)-(\alpha_n+\beta_n-2)\psi(\alpha_n+\beta_n)-\log \text{Beta}(\alpha_n,\beta_n) < C+\frac{1}{2}\log(n),
\end{align*}
for some large constant $C$.
\end{proof}

\begin{proposition}~\label{lem:kl-prior}
Let $\theta_0\in(0,1)$. Let, $\rho_n$ be a sequence of Beta distributions with parameters $\alpha_n=n\theta_0$ and $\beta_n=n(1-\theta_0)$. Let $\pi$ denote an Beta distribution, with parameters $(\alpha,\beta)$. Then, $\mathcal{K}(\rho_n,\pi)<C+\frac{1}{2}\log(n)$, for some constant $C>0$. 
\end{proposition}

\begin{proof}
The KL-divergence between $\rho_n$ and $\pi$ can be written as,
\begin{align*}
    \mathcal{K}(\rho_n,\pi)= & \int \log\left(\frac{\rho_n}{\pi}\right)\rho_n(d\theta)=\int \log\left(\frac{\rho_n}{U}\right)\rho_n(d\theta)+\int \log\left(\frac{U}{\pi}\right)\rho_n(d\theta),
\end{align*}
where, $U$ is an uniform distribution on $(0,1)$.
We analyze the second term in the above expression. The second term can be written as,
\begin{align*}
    \int \log\left(\frac{U}{\pi}\right)\rho_n(d\theta) & = \int \log \left( \frac{1}{\frac{1}{\text{Beta}(\alpha,\beta)}\theta^{\alpha-1}(1-\theta)^{\beta-1}} \right)\rho_n(d\theta)\\
     & = C_1-(\alpha-1)\int \log(\theta) \rho_n(d\theta)-(\beta-1)\int \log(1-\theta) \rho_n(d\theta),
     \intertext{where $C_1$ is $\log(\text{Beta}(\alpha,\beta))$. Since, $\rho_n$ follows a Beta distribution with parameters $\alpha_n=n\theta_0$ and $\beta_n=n(1-\theta_0)$, we get that,}
     \int \log\left(\frac{U}{\pi}\right)\rho_n(d\theta)& =
     C_1-(\alpha-1) \left[\psi(\alpha_n)-\psi(\alpha_n+\beta_n)\right]-(\beta-1)\left[\psi(\beta_n)-\psi(\alpha_n+\beta_n)\right]
\end{align*}
Since, $\log(x)-\frac{1}{x}<\psi(x)<\log(x)-\frac{1}{2x}$, looking at the term $\left[\psi(\alpha_n)-\psi(\alpha_n+\beta_n)\right]$, we get that, 
\begin{align*}
    \left[\psi(\alpha_n)-\psi(\alpha_n+\beta_n)\right] & = \left[\psi(n\theta_0)-\psi(n\theta_0+n(1-\theta_0))\right]\\
    & = \left[\psi(n\theta_0)-\psi(n)\right].\\
    \intertext{Using the upper bound on $\psi(n\theta_0)$ and the lower bound on $\psi(n)$, we get}
    \left[\psi(n\theta_0)-\psi(n)\right] & < \log(n\theta_0)-\frac{1}{2n\theta_0}-\log(n)+\frac{1}{n\theta_0}\\
    & = \log(\theta_0)+\frac{1}{2n\theta_0}.
    \intertext{We can also get a lower bound very similarly.}
    \left[\psi(\alpha_n)-\psi(\alpha_n+\beta_n)\right] & >\log(\theta_0)-\frac{1}{2n\theta_0}.
\end{align*}
Therefore it follows that
\begin{equation}
    \log(\theta_0)-\frac{1}{2n\theta_0}<\left[\psi(\alpha_n)-\psi(\alpha_n+\beta_n)\right]< \log(\theta_0)+\frac{1}{2n\theta_0}.
\end{equation}
Similarly, we have
\begin{equation}
    \log(1-\theta_0)-\frac{1}{2n(1-\theta_0)}<\left[\psi(\beta_n)-\psi(\alpha_n+\beta_n)\right]< \log(1-\theta_0)+\frac{1}{2n(1-\theta_0)}.
\end{equation}
Consequently, for a large enough constant $C_2>0$ it follows that
\begin{align*}
    -C_2<\left[\psi(\alpha_n)-\psi(\alpha_n+\beta_n)\right]< C_2,
\end{align*}
and
\begin{align*}
    -C_2<\left[\psi(\beta_n)-\psi(\alpha_n+\beta_n)\right]< C_2.
\end{align*}
Without loss of generality assume that $\min\{\alpha-1,\beta-1\}>0$. Then we get,
\begin{align*}
    -(\alpha-1)C_2<(\alpha-1)\left[\psi(\alpha_n)-\psi(\alpha_n+\beta_n)\right]< (\alpha-1)C_2,
\end{align*}
and
\begin{align*}
    -(\alpha-1)C_2<(\alpha-1)\left[\psi(\beta_n)-\psi(\alpha_n+\beta_n)\right]< (\alpha-1)C_2.
\end{align*}
If, either of $\alpha-1$ or $\beta-1$ was negative, then the direction of inequality on the previous equation would have reversed. However, due to the symmetric nature of the inequality, that would not have made a difference. Using the above bounds we finally show that, 
\begin{align*}
    C_1-(\alpha-1) \left[\psi(\alpha_n)-\psi(\alpha_n+\beta_n)\right]-(\beta-1)\left[\psi(\beta_n)-\psi(\alpha_n+\beta_n)\right]<C_1+C_2\theta_0+C_2(1-\theta_0)<C,
\end{align*}
for some large constant C. Finally, we upper bound $\int \log\left(\frac{\rho_n}{U}\right)\rho_n(d\theta)$ by \Cref{prop:kl-prior-unif} thereby completing the proof.
\end{proof}
\section{Proofs of Main Results}

\subsection{\bf Proposition~\ref{prop:1}}~\label{appendix:prop:1}
We start by recalling the variational formula of Donsker and Varadhan~\cite{donsker}.
\begin{lemma}[Donsker-Varadhan]~\label{lem:dv}
For any probability distribution function $\pi$ on $\Theta$, and for any measurable function $h: \Theta \rightarrow \mathbb{R}$, if $\int e^h d\pi <\infty$, then
\begin{align}
    \log\int e^h d\pi & = \underset{\rho\in \mathcal{M}^{+}(\Theta)}{\sup}\left\{\int h d\rho-\mathcal{K}(\rho,\pi) \right\}
\end{align}
\end{lemma}

\begin{proof}[Proof of Prop.~\ref{prop:1}]
Fix $\alpha\in (0,1)$, and $\theta\in\Theta$. First, observe that by the definition of the $\alpha$-R\'enyi divergence we have
\begin{align*}
     & \E^{(n)}_{\theta_0} [\exp (-\alpha \loglik)] = \exp[-(1-\alpha)\malpharen]\\
\intertext{Multiplying both sides of the equation by $\exp[(1-\alpha)\malpharen$ and integrating with respect to (w.r.t.) $\pi(\theta)$ it follows that}
     & \int \E^{(n)}_{\theta_0} \left[\exp \left(-\alpha \loglik+(1-\alpha)\malpharen \right) \right]\pi(d\theta)  = 1,
    \text{or}\\\qquad
     &  \E^{(n)}_{\theta_0}\left[\int \exp \left(-\alpha \loglik+(1-\alpha)\malpharen \right)\pi(d\theta) \right] = 1.
\end{align*}
Define $h(\theta) := -\alpha \loglik+(1-\alpha)\malpharen$. Then, applying \Cref{lem:dv} to the integrand on the left hand side (l.h.s.) above, it follows that
\begin{align*}
    \E^{(n)}_{\theta_0}\left[\exp\left( \underset{\rho\in\mathcal{M}^{+}(\Theta)}{\sup} \left[ \int h(\theta) \rho(d\theta)-\mathcal{K}(\rho,\pi) \right] \right)\right] = 1.
\end{align*}
Multiply both sides of this equation by $\epsilon > 0$ to obtain
\begin{align*}
 \E^{(n)}_{\theta_0}\left[\exp\left( \underset{\rho\in\mathcal{M}^{+}(\Theta)}{\sup} \left[ \int h(\theta) \rho(d\theta)-\mathcal{K}(\rho,\pi) + \log(\epsilon)\right] \right)\right] = \epsilon.
\end{align*}
Now, by Markov's inequality, we have
\begin{equation}
     {P^{(n)}_{\theta_0}}\bigg[ \underset{\rho\in\mathcal{M}^{+}(\Theta)}{\sup} \int(-\alpha \loglik+(1-\alpha)\malpharen)\rho(d\theta)-\mathcal{K}(\rho,\pi) + \log(\epsilon) \geq 0\bigg]\leq\epsilon.
\end{equation}
Thus, it follows via complementation that 
\begin{align*}
   {P^{(n)}_{\theta_0}}\bigg[ \forall \rho\in\mathcal{F}(\Theta) \int\malpharen\rho(d\theta)\leq\frac{\alpha}{(1-\alpha)} \int\loglik\rho(d\theta)+\frac{\mathcal{K}(\rho,\pi) - \log(\epsilon)}{1-\alpha} \bigg] & \geq 1-\epsilon,
\end{align*}
thereby completing the proof.
\end{proof}

\subsection{\bf \Cref{thm:pac-bayes}:}
\begin{proof}[Proof of Theorem~\ref{thm:pac-bayes}]
Recall the definition of the fractional posterior and the VB approximation,
\begin{align*}
      \pi_{n,\alpha|X^n} =\frac{\exp^{-\alpha r_n (\theta,\theta_0)(X^n)}\pi(d\theta)}{\int \exp^{-\alpha r_n (\gamma,\theta_0)(X^n)}\pi(d\gamma)},~
      \Tilde{\pi}_{n,\alpha|X^n} =\underset{\mathcal{\rho \in F}}{\arg \min}~ \mathcal{K}(\rho,\pi_{n,\alpha|X^{(n)}}).
\end{align*}

It follows by definition of the KL divergence that
\begin{align}
    \Tilde{\pi}_{n,\alpha|X^n} & =\underset{\mathcal{\rho \in F}}{\arg \min}\left\{ -\alpha \int \loglik \rho(d\theta) +\mathcal{K}(\rho,\pi) \right\},
\end{align}
where $\pi$ is the prior distribution. 
Following~\Cref{prop:1} it follows that for any $\epsilon > 0$ 
\begin{equation*} 
\int\malpharen {\Tilde{\pi}(d\theta|X^n)}\leq\frac{\alpha}{(1-\alpha)} \int\loglik\rho(d\theta)+\frac{\mathcal{K}(\rho,\pi) - \log(\epsilon)}{1-\alpha},
\end{equation*}
with probability $1-\epsilon$. 
We fix an $\eta \in (0,1)$.  
Using Chebychev's inequality, we have
\begin{align*}
  & P_{\theta_0}^{(n)}\left[ \frac{\alpha}{1-\alpha} \int \loglik\rho_n(d\theta) \geq  \frac{\alpha}{1-\alpha}\int \E[\loglik]\rho_n(d\theta)+\frac{\alpha}{1-\alpha} 
    \sqrt{ \frac{\Var[\int \loglik \rho_n(d\theta)]}{\eta} } +\frac{\mathcal{K}(\rho_n,\pi)}{1-\alpha} \right] \\
 &= P_{\theta_0}^{(n)}\left[ \frac{\alpha}{1-\alpha} \int \loglik\rho_n(d\theta) -  \frac{\alpha}{1-\alpha} \int \E[\loglik]\rho_n(d\theta) - \frac{\mathcal{K}(\rho_n,\pi)}{1-\alpha}  \geq 
   \frac{\alpha}{1-\alpha}
\sqrt{ \frac{\Var[\int \loglik \rho_n(d\theta)]}{\eta} }\right]\\
  & \leq \frac{\Var\left[ \frac{\alpha}{1-\alpha} \int \loglik\rho_n(d\theta) -  \frac{\alpha}{1-\alpha}\int \E[\loglik]\rho_n(d\theta) - 
 \frac{\mathcal{K}(\rho_n,\pi)}{1-\alpha}    \right]}{\frac{\alpha^2}{(1-\alpha)^2}
 \frac{\Var[\int \loglik \rho_n(d\theta)]}{\eta} }. \\
 \end{align*}
 Note that $\frac{\alpha}{1-\alpha}\int E(\loglik)\rho_n(d\theta)$ and $\frac{\mathcal{K}(\rho_n,\pi)}{1-\alpha}$ are constants with respect to the data, implying
 \begin{align*}
    \Var\bigg[ \frac{\alpha}{1-\alpha} &\int \loglik\rho_n(d\theta) -  \frac{\alpha}{1-\alpha}\int \E[\loglik]\rho_n(d\theta) - \frac{\mathcal{K}(\rho_n,\pi)}{1-\alpha}\bigg]
    \\& \qquad= \frac{\alpha^2}{(1-\alpha)^2}
 {\Var\left[\int \loglik \rho_n(d\theta)\right]}.
\end{align*}
Therefore, we have
\begin{align*}
    P_{\theta_0}^{(n)}\bigg[ \frac{\alpha}{1-\alpha} \int \loglik\rho_n(d\theta) \geq  \frac{\alpha}{1-\alpha}&\int \E[\loglik]\rho_n(d\theta)\\& +\frac{\alpha}{1-\alpha} 
    \sqrt{ \frac{\Var[\int \loglik \rho_n(d\theta)]}{\eta} } +\frac{\mathcal{K}(\rho_n,\pi)}{1-\alpha} \bigg] \leq \eta.
\end{align*}
From \Cref{prop:1}, with probability $1-\epsilon$ the following holds 
\begin{equation*}
     \int \malpharen \Tilde{\pi}_{n,\alpha|X^n}(d\theta) \leq \frac{\alpha\int \loglik \rho_n(d\theta)+\mathcal{K}(\rho_n,\pi)-\log(\epsilon)}{1-\alpha}.
\end{equation*}
Therefore, with probability $1-\eta-\epsilon$ the following statement holds
\begin{align}~\label{eq:proof-step-miss}
    \int \malpharen \Tilde{\pi}_{n,\alpha|X^n}(d\theta) &\leq \frac{\alpha}{1-\alpha}\int \kld  \rho_n(d\theta)+\frac{\alpha}{1-\alpha} \sqrt{ \frac{\text{Var}[\int \loglik \rho_n(d\theta)]}{\eta} }\nonumber\\ & \qquad \qquad +\frac{\mathcal{K}(\rho_n,\pi)-\log(\epsilon)}{1-\alpha}.
\end{align}
Next, observe that
\begin{align*}
    \text{Var}\left[\int \loglik \rho_n(d\theta)\right] = E_{\theta_0}^{(n)} \left[ \left| \int \loglik \rho_n(d\theta) - E\left[ \int \loglik \rho_n(d\theta)\right] \right|^2 \right]\leq \int \text{Var}[\loglik ] \rho_n(d\theta),
\end{align*}
by a straightforward application of Jensen's inequality to the inner integral on the left hand side. Finally, following the hypotheses~(i),~(ii) and~(iii), we have,
\begin{align*}
    \int \malpharen \Tilde{\pi}_{n,\alpha|X^n}(d\theta) &\leq  \frac{\alpha}{1-\alpha}\int \bigg( \kld + \sqrt{ \frac{\int\Var[ \loglik ]\rho_n(d\theta)}{\eta} }\bigg)\rho_n(d\theta)\\  & \qquad  +\frac{1}{\alpha}\left(\mathcal{K}(\rho_n,\pi)-\log(\epsilon)\right)\\
    & \leq  \frac{\alpha(\epsilon_n+\sqrt{\frac{n\epsilon_n}{\eta}})}{1-\alpha}+\frac{n\epsilon_n-\log(\epsilon)}{1-\alpha}, 
\end{align*}
thereby concluding the proof.
\end{proof} 
\textbf{Lemma 2 :}
\subsection{\Cref{prop:1a}}
\begin{proof}[Proof of Proposition~\ref{prop:1a}]
We define $Y_i :=  \log\left(\frac{p_{\theta_1}(X_i|X_{i-1})}{p_{\theta_2}(X_i|X_{i-1})} \right)$ for $i = 1,\ldots,n$, and $Z_0=\log \bigg(\frac{q^{(0)}_{1}(X_0)}{q^{(0)}_2(X_0)}\bigg)$. Then, using the Markov property we can see that the Kullback-Leibler divergence between the joint distributions $P_{\theta_1}^{(n)}$ and $P_{\theta_2}^{(n)}$ satisfies
\(
  \mathcal{K}\left(P_{\theta_1}^{(n)},P_{\theta_2}^{(n)}\right) = \sum_{i=1}^n \E_{\theta_1} \left[ Y_i \right]+\E_{\theta_1}[Z_0].
\)
If the Markov chain $\{X_i\}$ is stationary under $\theta_1$, so is $\{Y_i\}$. Hence $Y_i\overset{d}{=}Y_1$ and the above equation reduces to, 
\begin{align}
    \mathcal{K}\left(P_{\theta_1}^{(n)},P_{\theta_2}^{(n)}\right) = n\E_{\theta_1} \left[ Y_1 \right]+\E_{\theta_1}[Z_0].
\end{align}
\end{proof}

\subsection{\Cref{prop:2}}
First, recall the following result from \cite{ibrag}.
\begin{lemma}{\cite[Lemma 1.2]{ibrag}}~\label{lem:alpha-mix}
Let $X_{-\infty},\dots,X_1,X_2,\dots$ be an alpha mixing Markov chain with alpha mixing coefficients given by $\alpha_k$. 
Let $\mathcal{M}_{a}^{b}$ be the sigma-field generated by the subsequence $(X_a,X_{a+1},\dots,X_b)$. Let $\eta_t\in \lowersigma$ and $\tau_t\in \uppersigma$ be adapted random variables such that $|\eta_t|\leq1,|\tau_t|\leq1$. Then,
\begin{align}
\underset{t}{\sup} \underset{\eta_t ,\tau_t}{\sup}|\E[\eta_t\tau_t]-\E[\eta_t] \E[\tau_t]|\leq 4\alpha_k.
\end{align}
\end{lemma}
This lemma provides an upper bound on the covariance of events $\eta \text{ and } \tau$, as shown next.

\begin{lemma}~\label{lem:corr-bound}
Let $\eta\in \lowersigma$ $\tau\in \uppersigma$ be such that, $E|\eta|^{2+\delta}\leq C_1,E|\tau|^{2+\delta}\leq C_2 \text{ for some } \delta>0$. Then, for a fixed $n < +\infty$, we have
\begin{equation}
 |\E\eta\tau-\E\eta \E\tau| \leq
    \left(\frac{4}{n} +
     2n^{\delta/2}(C_1+C_2)+
     2n^{\delta/2}\sqrt{C_1C_2}\right)\alpha_k^{2\delta/(2+\delta)}.
     \end{equation}
\end{lemma}

\begin{proof}
     Let $N < +\infty$ be a fixed number. We get from the triangle inequality that
    \begin{align}
    |\E\eta\tau-\E\eta \E\tau| & \leq |\E\eta\tau I_{[|\eta| \leq N,|\tau|\leq N]}-\E\eta I_{[|\eta|\leq N]} \E\tau I_{[|\tau|\leq N]}| \\\nonumber
     &\qquad +|\E\eta\tau I_{[|\eta| \geq N,|\tau|\leq N]}-\E\eta I_{[|\eta|\geq N]} \E\tau I_{[|\tau|\leq N]}| \\\nonumber
     &\qquad +|\E\eta\tau I_{[|\eta| \leq N,|\tau|\geq N]}-\E\eta I_{[|\eta|\leq N]} \E\tau I_{[|\tau|\geq N]}|\\\nonumber
     &\qquad +|\E\eta\tau I_{[|\eta| \geq N,|\tau|\geq N]}-\E\eta I_{[|\eta|\geq N]} \E\tau I_{[|\tau|\geq N]}|.
     \intertext{Multiplying and dividing the first term by $N^2$ and applying \Cref{lem:alpha-mix}, we get $|\E\eta\tau I_{[|\eta| \leq N,|\tau|\leq N]}-\E\eta I_{[|\eta|\leq N]} \E\tau I_{[|\tau|\leq N]}|\leq 4N^2 \alpha_k$. 
     For the second term, if $|\tau|\leq N$, then $\tau\leq N$ and $\tau\geq -N$. Plugging this in the second term we get,}
     |\E\eta\tau I_{[|\eta| \geq N,|\tau|\leq N]} - & \E\eta I_{[|\eta|\geq N]} \E\tau I_{[|\tau|\leq N]}|  \leq 
     \left|N\E\eta  I_{[|\eta| \geq N}+N\left[\E\eta I_{[|\eta|\geq N]} \right]\right| \\
     & = 2N |\E\eta I_{[|\eta|\geq N]} |.
     \intertext{Since $|\eta|\geq N$,  we have $1\leq \frac{|\eta|^{1+\delta}}{N^{1+\delta}}$. Following this,}
     |2N\E\eta I_{[|\eta|\geq N]} | & \leq 2N\left|\E\left[ \frac{|\eta|^{2+\delta}}{N^{1+\delta}} I_{[|\eta|\geq N]}\right]\right| \\
     & \leq 2N\frac{1}{N^{1+\delta}}|\E \eta^{2+\delta}|
     \leq 2\frac{C_1}{N^{\delta}}.
     \end{align}
     Similarly, we can also write for the third term, $|\E\eta\tau I_{[|\eta| \leq N,|\tau|\geq N]}-\E\eta I_{[|\eta|\leq N]} \E\tau I_{[|\tau|\geq N]}|\leq 2\frac{C_2}{N^{\delta}}$. Finally, for the last term we get that by Cauchy-Schwarz inequality,
     \begin{align}
        |\E\eta\tau I_{[|\eta| \geq N,|\tau|\geq N]}-\E\eta I_{[|\eta|\geq N]} \E\tau I_{[|\tau|\geq N]}| & \leq \sqrt{\Var\left[\eta I_{[|\eta| \geq N]} \right]\Var\left[\tau I_{[|\tau|\geq N]} \right]}\\
        & < 2\sqrt{\Var\left[\eta I_{[|\eta| \geq N]} \right]\Var\left[\tau I_{[|\tau|\geq N]} \right]}\\
        & \leq 2 \sqrt{\E\left[\eta^2 I_{[|\eta| \geq N]} \right]\E\left[\tau^2 I_{[|\tau|\geq N]} \right]}.
     \end{align}
     Since $|\eta|>N$, $1<\frac{|\eta|^{\delta}}{N^{\delta}}$. Similarly, $1<\frac{|\tau|^{\delta}}{N^{\delta}}$. Plugging these in the previous equation, we get,
     \begin{align}
         \sqrt{\E\left[\eta^2 I_{[|\eta| \geq N]} \right]\E\left[\tau^2 I_{[|\tau|\geq N]} \right]}& \leq \sqrt{\frac{1}{N^{2\delta}}\E\left[|\eta|^{2+\delta} I_{[|\eta| \geq N]} \right]\E\left[|\tau|^{2+\delta} I_{[|\tau|\geq N]} \right] }\\
         & \leq \frac{1}{N^{\delta}} \sqrt{C_1C_2}.
     \end{align}
     Combining the four upper bounds above, we get,
     \begin{align}
     |\E\eta\tau-\E\eta \E\tau| & \leq 4N^2\alpha_k+\frac{2}{N^{\delta}}(C_1+C_2)+\frac{2}{N^{\delta}}\sqrt{C_1 C_2}.
     \end{align}
Now, in particular, setting $N=n^{-1/2}\alpha_k^{-1/(2+\delta)}$ it follows that
    \begin{align}
     |\E\eta\tau-\E\eta \E\tau| & \leq
     \frac{4}{n}\alpha_k^{\delta/(2+\delta)}+2n^{\delta/2}\alpha_k^{\delta/(2+\delta)}(C_1+C_2)+2n^{\delta/2}\alpha_k^{\delta/(2+\delta)}\sqrt{C_1C_2}\\ 
     & = \bigg(\frac{4}{n} +
     2n^{\delta/2}(C_1+C_2)+2n^{\delta/2}\sqrt{C_1C_2}\bigg)\alpha_k^{\delta/(2+\delta)}.
    \end{align}
\end{proof}

\begin{lemma}\label{lem:corr-bound2}
Let $\{X_t\}$ be an alpha mixing Markov chain with mixing coefficient $\alpha_k$. Further assume that $\E|X_t|^{2+\delta}\leq C_1 \text{ and } \E|X_{t+k}|^{2+\delta}\leq C_2$ for some $\delta>0$.
Then, for any $t$ and any $n>0$
\begin{equation}\label{ibr}
    |\Cov(X_t,X_{t+k})| \leq 
    \bigg(\frac{4}{n} +
     2n^{\delta/2}(C_1+C_2)+
     2n^{\delta/2}\sqrt{C_1C_2}\bigg)\alpha_k^{\delta/(2+\delta)}.
\end{equation}
\end{lemma}
\begin{proof}
Set $\eta=X_t, \tau=X_{t+k}$ in \Cref{lem:corr-bound}.
\end{proof}
\begin{proof}[Proof of \Cref{prop:2}]
 Let $\{X_t\}$ be a stationary alpha-mixing Markov chain under $\theta_1$ with mixing coefficients $\{\alpha_k\}$. Observe that the log-likelihood can be expressed as
\begin{align*}
   r_n(\theta_2,\theta_1) & =\underset{i=1}{\sum^{n}} \log \bigg( \frac{ p_{\theta_1}(X_i|X_{i-1}) } { p_{\theta_2}(X_i|X_{i-1}) } \bigg) + \log \bigg(\frac{q^{(0)}_{1}(X_0)}{q^{(0)}_2(X_0)}\bigg)\\
     & \equiv \underset{i=1}{\sum^n} Y_i+Z_0.
\end{align*}
Therefore, the variance of the log-likelihood ratio is simply
\begin{align*}
  \Var_{\theta_1} \left[ r_n(\theta_2,\theta_1) \right] & =\Var_{\theta_1}\left[\underset{i=1}{\sum^n} Y_i+Z_0\right]\\
     & =\underset{i,j=1}{\sum^n} \Cov_{\theta_1}(Y_i,Y_j) + \underset{i=1}{\sum^n} \Cov_{\theta_1}(Y_i,Z_0) + \Cov_{\theta_1}(Z_0,Z_0).
\end{align*}
Now, using \Cref{lem:corr-bound2} we have
\begin{align*}
    |\Cov_{\theta_1}(Y_i,Y_j)| & =|\E_{\theta_1} Y_iY_j-\E_{\theta_1} Y_i\E_{\theta_1} Y_j|\\
     & < \left(\frac{4}{n}+2n^{\delta/2}(\E_{\theta_1}|Y_i|^{2+\delta}+\E_{\theta_1}|Y_j|^{2+\delta}+\sqrt{\E_{\theta_1}|Y_i|^{2+\delta}\E_{\theta_1}|Y_j|^{2+\delta}})\right)\alpha_{|j-i|-1}^{\delta/(2+\delta)}\\
     & = \left(\frac{4}{n}+2n^{\delta/2}(C_{\theta_1,\theta_2}^{(i)}+C_{\theta_1,\theta_2}^{(j)}+\sqrt{C_{\theta_1,\theta_2}^{(i)}C_{\theta_1,\theta_2}^{(j)}})\right)\alpha_{|j-i|-1}^{\delta/(2+\delta)}.\\
     \intertext{Similarly, as above we can also say}
    |\Cov_{\theta_1}(Y_i,Z_0)|< & \left(\frac{4}{n}+2n^{\delta/2}(C_{\theta_1,\theta_2}^{(i)}+D_{1,2}+\sqrt{C_{\theta_1,\theta_2}^{(i)}D_{1,2}})\right)\left(\alpha_{i-1}^{\delta/(2+\delta)}\right)
\end{align*}
Combining, the two upper bounds above, we get the first result:
\begin{align*}
    \Var_{\theta_1} \bigg[ r_n(\theta_2,\theta_1) \bigg] & < \sum_{i,j=1}^n\left(\frac{4}{n}+2n^{\delta/2}(C_{\theta_1,\theta_2}^{(i)}+C_{\theta_1,\theta_2}^{(j)}+2\sqrt{C_{\theta_1,\theta_2}^{(i)}C_{\theta_1,\theta_2}^{(j)}})\right)  \left(\alpha_{|i-j|-1}^{\delta/(2+\delta)}\right)\\
    \nonumber
     &\qquad+ \sum_{i=1}^n\left(\frac{4}{n^2}+2n^{\delta/2}(C_{\theta_1,\theta_2}^{(i)}+D_{1,2}+\sqrt{C_{\theta_1,\theta_2}^{(i)}D_{1,2}})\right)\left(\alpha_{i-1}^{\delta/(2+\delta)}\right) + \Var[Z_0,Z_0].
\end{align*}
If $\{X_i\}$ is stationary under $\theta_1$, so is $\{Y_i\}$. Therefore, $\E_{\theta_1} |Y_i|^{2+\delta}= \E_{\theta_1} |Y_1|^{2+\delta}=C_{\theta_1,\theta_2}^{(1)}\enspace\forall\enspace i$, and
\begin{align}
\nonumber
    \underset{i,j=1}{\sum^n} \Cov_{\theta_1} (Y_i,Y_j) & \leq\underset{i,j=1}{\sum^n} \left(\frac{4}{n}+6n^{\delta/2}C_{\theta_1,\theta_2}^{(1)}\right)\alpha_{|j-i|-1}^{\delta/(2+\delta)}\\
    \label{eq:cross-bound}
     & \leq n\,\left(\frac{4}{n}+6n^{\delta/2}C_{\theta_1,\theta_2}^{(1)} \right)\left(\underset{h\geq 1}{\sum} \alpha_{h-1}^{\delta/(2+\delta)}\right).
\end{align}
Again, using \Cref{lem:corr-bound2} on $\Cov_{\theta_1}(Y_i,Z_0)$, yields
\begin{equation}
\label{eq:init-bound}
\underset{i=1}{\sum^n} \Cov_{\theta_1}(Y_i,Z_0) \leq \left(\frac{4}{n}+2n^{\delta/2}(C_{\theta}+D_{1,2}+\sqrt{C_{\theta}D_{1,2}})\right)\left(\sum_{h\geq 1} \alpha_{h}^{\delta/(2+\delta)}\right).
\end{equation}
Finally, using ~\cref{eq:cross-bound} and~\cref{eq:init-bound} we have
\begin{align*}
    \Var_{\theta_1} \left[ r_n(\theta_2,\theta_1) \right] & \leq n \left(\frac{4}{n}+6n^{\delta/2}C_{\theta_1,\theta_2}^{(1)}\right) \left(\sum_{h\geq 1}\alpha_{h-1}^{\delta/(2+\delta)}\right) + \\
    &\left(\frac{4}{n}+2n^{\delta/2}(C_{\theta_1,\theta_2}^{(1)}+D_{1,2}+\sqrt{C_{\theta_1,\theta_2}^{(1)}D_{1,2}}) \right) \left(\sum_{h\geq 1} \alpha_{h}^{\delta/(2+\delta)}\right) + \Cov_{\theta_1} (Z_0,Z_0).
\end{align*}
\end{proof}
\begin{lemma}\label{lem:paired-meas}
{Let $\{X_t\}$ be an alpha-mixing Markov Chain with mixing coefficients $\{\alpha_t\}$. Then the process $\{Y_t\}$ where $Y_t := \log \left( \frac{ p_{\theta_0}(X_t|X_{t-1}) } { p_{\theta}(X_t|X_{t-1}) } \right)$} is also alpha-mixing with mixing coefficients $\{\tilde \alpha_t\} $ where $\tilde \alpha_t = \alpha_{t-1}$.
\end{lemma}
\begin{proof}
By $Z_i$ denote the paired random measure $(X_i,X_{i-1})$. Let $\mathcal{M}_i^j$ denote the sigma field generated by the measures $X_k, \text{ where } i\leq k\leq j$. By $\mathcal{G}_i^j$ denote the sigma field generated by the measures $Z_k, \text{ where } i\leq k\leq j$. Let $C\in \mathcal{M}_{i-1}^j$. Then, $C$ can be expressed as $(C_{i-1}\times C_i\times \dots \times C_j)$. for $C_{i-1}\in \mathcal{M}_{i-1}^{i-1}\text{, } C_i \in \mathcal{M}_i^i \dots$ and so on. Now, consider a map. $T_i^j : (C_{i-1}\times C_i\times \dots \times C_j)\xrightarrow{} (C_{i-1}\times C_i\times C_i \times \dots \times C_{j-1} \times C_{j-1} \times C_j)$. Note that, $T(C)\in \mathcal{G}_i^j$. It is easy to see that $\mathcal{G}_i^j = T_i^j(\mathcal{M}_{i-1}^j)\cup \mathcal{M}_{i-1}^{*j}$, where $T_i^j(\mathcal{M}_{i-1}^j)$ is obtained by applying the map $T_i^j$ to each element of $\mathcal{M}_{i-1}^j$. If we assume this latter set to be the range and $\mathcal{M}_{i-1}^j$ to be the domain, then, by construction, $T_i^j$ is a bijection. Also, the two classes are made of disjoint sets, i.e. if $A\in T_i^j(\mathcal{M}_{i-1}^j) \text{ and } A^*\in \mathcal{M}_{i-1}^{*j}$, then $A\cap A^*=\phi$. Also, note that $\mathcal{M}_{i-1}^{j*}$ is made of impossible sets. i.e. $P(A^*)=0 \enspace \forall\enspace A^* \in \mathcal{M}_{i-1}^{j*}$. Now consider the alpha mixing coefficients for $Z_i$. By definition, it is given by
\begin{align*}
  \alpha^z_k & = \underset{i}{\sup}\underset{A\in\mathcal{G}_{-\infty}^{i} ,B\in\mathcal{G}_{i+k}^{\infty}}{\sup}|P(A\cap B)-P(A)P(B)|\\
   & = \underset{i}{\sup}\underset{A\in\mathcal{G}_{-\infty}^{i} ,B\in\mathcal{G}_{i+k}^{\infty}}{\sup}|P((A^o\cup A^*)\cap (B^o\cup B^*))-P((A^o\cup A^*))P((B^o\cup B^*))|.
\end{align*}
Where,
\begin{center}
\begin{tabular}{l l}
$A=(A^o\cup A^*)$ & $B=(B^o\cup B^*)$ \\ 
$A^o \in \mathcal T_{-\infty}^i({M}_{-\infty}^{i})$ & $A^* \in \mathcal{M}_{-\infty}^{*i}$ \\  
$B^o \in T_{i+k-1}^{\infty}(\mathcal{M}_{j+k-1}^{\infty})$ & $B^* \in \mathcal{M}_{j+k-1}^{*\infty}.$    
\end{tabular}
\end{center}
Then, the expression for the alpha mixing coefficient can be reduced into
\begin{align*}
    \alpha^z_k & = \underset{i}{\sup}\underset{A^o \in T_{-\infty}^{i}(\mathcal{M}_{-\infty}^{i}) ,B^o\in T_{i+k-1}^{\infty}(\mathcal{M}_{i+k-1}^{\infty})}{\sup}|P(A^o\cap B^o)-P(A^o)P(B^o)|.
\end{align*}
Note that, by bijection property of $T_i^j$, we can find $A' \in \mathcal{M}_{-\infty}^{i}$ and $B' \in \mathcal{M}_{i+k-1}^{\infty}$ such that
\begin{align*}
     \alpha^z_k & = \underset{i}{\sup}\underset{A' \in \mathcal{M}_{-\infty}^{i} ,B'\in \mathcal{M}_{i+k-1}^{\infty}}{\sup}|P(T_{-\infty}^{i}(A')\cap T_{i+k-1}^{\infty}(B'))-P(T_{-\infty}^{i}(A'))P(T_{i+k-1}^{\infty}(B'))|.\\
      & = \alpha_{k-1}.
\end{align*}
Now, $\log \left( \frac{ p_{\theta_0}(X_n|X_{n-1}) } { p_{\theta}(X_n|X_{n-1}) } \right)$ is just a function of the paired Markov chain $Z_i$, therefore it has alpha-mixing coefficient $\alpha_{k-1}$.
\end{proof}
\subsection{Proof of \Cref{thm:lip-gen} : }
\begin{proof} 
\textit{Part 1: Verifying condition (i) of \Cref{thm:pac-bayes2}.}

We substitute the true parameter $\theta_0$  for $\theta_1$ and $\theta$ for $\theta_2$. 
We also set $q_{1}^{(0)}$ to be the invariant distribution of the Markov chain under $\theta_0$, $q_{0}$, and $q_{2}^{(0)}$ as the invariant distribution of the Markov chain under $\theta$, $q_{\theta}$. Applying the fact that these Markov chains are stationary to \Cref{prop:1a}, we have
\begin{align}
  \kld 
   & = n \E \left[ \log\left( \frac{p_{\theta_0}(X_{1}|X_{0})}{p_{\theta}(X_{1}|X_{0})} \right) \right]+\E[Z_0],\nonumber\\
   & \leq n\sum_{j=1}^m \E\left[ M^{(1)}_j(X_1,X_{0}) \right] |f^{(1)}_j(\theta,\theta_0)|+\sum_{k=1}^m{\E [M^{(2)}_k(X_0)]|f^{(2)}_k(\theta,\theta_0)|},\label{eq:kld-bound-1}
\end{align}
where the inequality follows from Assumption~\ref{assume:gen-lip}. Therefore, it follows that
\begin{align*}
     \int \kld \rho_n(d\theta) & \leq n\sum_{j=1}^m \E\left[M^{(1)}_j(X_1,X_0)\right] \int |f^{(1)}_j(\theta,\theta_0)|\rho_n(d\theta)+\sum_{k=1}^m\E [M^{(2)}_k(X_0)]| \int f^{(2)}_k(\theta,\theta_0)|  \rho_n(d\theta).\\
     \intertext{By \cref{assume:gen-lip(0)} in \Cref{{assume:gen-lip}}, it follows that}
    \int \kld \rho_n(d\theta) & \leq n\sum_{j=1}^m \E\left[M^{(1)}_j(X_1,X_0)\right]\frac{C}{\sqrt{n}}+\sum_{k=1}^m\E [M^{(2)}_k(X_0)]\frac{C}{\sqrt{n}} \leq n \epsilon_n^{(1)},
\end{align*}
where $\epsilon_n^{(1)} \in O\left( \frac{1}{\sqrt{n}} \right)$.\\
\textit{Part 2: Verifying condition (ii) of \Cref{thm:pac-bayes2}.}
Again, using \Cref{prop:2} along with the fact that the Markov chain is stationary we have
\begin{align*}
    \Var[\loglik] & \leq  n\left(\frac{4}{n}+6n^{\delta/2}C_{\theta_0,\theta}^{(1)}\right)  \left(\sum_{k\geq 0}\alpha_{k}^{\delta/(2+\delta)}\right)\\
    \nonumber
     &\qquad+ \left(\frac{4}{n^2}+2n^{\delta/2}(C_{\theta_0,\theta}^{(1)}+D_{\theta_0,\theta}+\sqrt{C_{\theta_0,\theta}^{(1)}D_{\theta_0,\theta}})\right)\left(\sum_{k\geq 1} \alpha_k^{\delta/(2+\delta)}\right) + \Var[Z_0].
     \intertext{It then follows that}
      \int \Var[\loglik] \rho_n (d\theta) & \leq  n\left(\frac{4}{n}+6n^{\delta/2}\int C_{\theta_0,\theta}^{(1)}\rho_n (d\theta)\right)\left(\sum_{k\geq 1}\alpha_{k-1}^{\delta/(2+\delta)}\right) + \int \Var[Z_0]\rho_n (d\theta) + \\
     \nonumber
     &\hspace{-.7in} \left(\frac{4}{n^2}+2n^{\delta/2}(\int C_{\theta_0,\theta}^{(1)}\rho_n (d\theta)+\int D_{\theta_0,\theta}\rho_n (d\theta)+\int \sqrt{C_{\theta_0,\theta}^{(1)}D_{\theta_0,\theta}}\rho_n (d\theta))\right)
     \left(\sum_{k\geq 1} \alpha_k^{\delta/(2+\delta)}\right).
\end{align*}
First, consider the term $\int C_{\theta_0,\theta}^{(1)}\rho_n (\theta)$, and observe that
\begin{align*}
    \int C_{\theta_0,\theta}^{(1)}\rho_n (d\theta) & = \int \E \log \left|\frac{p_{\theta_0}(X_1|X_0)}{p_{\theta}(X_1|X_0)}\right|^{2+\delta} \rho_n(d\theta).
    \intertext{By \Cref{assume:gen-lip}, we have}
    \int \E \log \left|\frac{p_{\theta_0}(X_1|X_0)}{p_{\theta}(X_1|X_0)}\right|^{2+\delta} \rho_n(d\theta) & \leq \int \E \left[\sum_{j=1}^m M^{(1)}_j(X_1,X_0)|f^{(1)}_k(\theta,\theta_0)|\right]^{2+\delta} \rho_n(d\theta).\\
    \intertext{Since the function $x \mapsto x^{2+\delta}$ is convex, we can apply Jensen's inequality to obtain,}
    \left(\sum_{j=1}^m M^{(1)}_j(X_1,X_0)|f^{(1)}_k(\theta,\theta_0)|\right)^{2+\delta} & \leq m^{1+\delta}\sum_{k=1}^m M^{(1)}_j(X_1,X_0)^{2+\delta}|f^{(1)}_k(\theta,\theta_0)|^{2+\delta}.
    \intertext{Therefore, it follows that}
    \int \E \log \left|\frac{p_{\theta_0}(X_1|X_0)}{p_{\theta}(X_1|X_0)}\right|^{2+\delta} \rho_n(d\theta) & \leq m^{1+\delta}\sum_{k=1}^m \E  [M^{(1)}_k(X_1,X_0)^{2+\delta}]\int |f^{(1)}_k(\theta,\theta_0)|^{2+\delta}\rho_n(d\theta).
    \intertext{By \Cref{assume:gen-lip}, $\int |f_k(\theta,\theta_0)|^{2+\delta}\rho_n(d\theta)<\frac{C}{n}$ and $\E [ M^{(1)}_k(X_1,X_0)^{2+\delta}]<B$, implying that}
     \int C_{\theta_0,\theta}^{(1)}\rho_n (d\theta)  & \leq m^{1+\delta}\sum_{k=1}^m B\frac{C}{n} = m^{2+\delta}\frac{BC}{n}.
\end{align*}
Since $\left(\sum_{k\geq 0}\alpha_{k}^{\delta/(2+\delta)}\right)<\infty$, 
it follows that
    $\left(\frac{4}{n}+6n^{\delta/2}\int C_{\theta_0,\theta}^{(1)}\rho_n(d\theta)\right)  \left(\sum_{k\geq 1}\alpha_{k-1}^{\delta/(2+\delta)}\right)\in \mathrm{O}(\frac{n^{\delta/2}}{n})$.
\sloppy Similarly, we can show that 
$\int D_{\theta_0,\theta}\rho_n (d\theta) \in O(\frac{1}{n})$, and $\int \Var[Z_0] \rho_n(d\theta) \in O(\frac{1}{n})$.

For the final term $\int \sqrt{C_{\theta_0,\theta}^{(1)}D_{\theta_0,\theta}}\rho_n (d\theta)$, use the Cauchy-Schwarz inequality to obtain the upper bound $\left(\int {C_{\theta_0,\theta}^{(1)}\rho_n (d\theta) \int D_{\theta_0,\theta}}\rho_n (d\theta)\right)^{1/2}$ which is also of order $O(\frac{1}{n})$.
Combining all of these together we have
\begin{align*}
    \int \Var[\loglik] \rho_n (d\theta) \leq n \epsilon^{(2)}_n, 
\end{align*}
for some $\epsilon^{(2)}_n\in \mathrm{O}(\frac{n^{\delta/2}}{n})$.
{
}
Since
\(
    \mathcal{K}(\rho_n,\pi) <\sqrt{n}C = n \frac{C}{\sqrt{n}},
\)
it follows that
\(
    \mathcal{K}(\rho_n,\pi) <n \epsilon_n^{(3)},
\)
where $\epsilon_n^{(3)} = \mathrm{O}(1/\sqrt{n})$ as before. Finally, by choosing $\epsilon_n=\max(\epsilon_n^{(1)},\epsilon_n^{(2)},\epsilon_n^{(3)})$, our theorem is proved.
\end{proof}

\subsection{Proof of \Cref{thm:bound-func}}
\begin{proof} 
\textit{Verifying condition (i) of \Cref{thm:pac-bayes2}:}
As in the proof of \Cref{thm:lip-gen} substitute the true parameter $\theta_0$  for $\theta_1$ and $\theta$ for $\theta_2$ in . We also set $q_{1}^{(0)}$ and $q_{2}^{(0)}$ to the distribution $q^{(0)}$. Applying \Cref{prop:1a} to the corresponding transition kernels and initial distribution we have,
\begin{align}
  \kld & = \sum_{i=1}^n \E\left[ \log \left( \frac{p_{\theta_0}(X_i|X_{i-1})}{p_{\theta}(X_i|X_{i-1})}\right) \right]+ \E\left[\log\left(\frac{D(X_0)}{D(X_0)}\right)\right]\\
  \nonumber
  &=\sum_{i=1}^n \E \left[ \log \left( \frac{p_{\theta_0}(X_i|X_{i-1})}{p_{\theta}(X_i|X_{i-1})}\right) \right].
\end{align}
Now, applying Assumption~\ref{assume:gen-lip}, we can bound the previous equation as follows,
\begin{align}~\label{equation:kld-upper-bound-ns}
    \kld & \leq \sum_{i=1}^n \E \left[ \sum_{k=1}^{m} M^{(1)}_k(X_i,X_{i-1})|f^{(1)}_k(\theta,\theta_0)| \right]\nonumber\\
    & = \sum_{i=1}^n  \sum_{k=1}^{m} \E\left[ M^{(1)}_k(X_i,X_{i-1}) \right]|f^{(1)}_k(\theta,\theta_0)|.
\end{align}
Since $M^{(1)}_k$'s are bounded there exists a constant $Q$ so that,
\begin{align*}
    \int \kld \rho_n(d\theta) & \leq Q\int \sum_{i=1}^n  \sum_{k=1}^{m}  |f^{(1)}_k(\theta,\theta_0)|\rho_n(d\theta)\\
    & = Q n\sum_{k=1}^{m} \int |f^{(1)}_k(\theta,\theta_0)|\rho_n(d\theta).
    \intertext{By~\cref{assume:gen-lip(2)} in \Cref{assume:gen-lip}, it follows that}
    \int \kld \rho_n(d\theta) & \leq Q n\sum_{k=1}^{m} \frac{C}{\sqrt{n}}
     = nmQ\frac{C}{\sqrt{n}}
     = n\epsilon^{(1)}_n,
\end{align*}
for some $\epsilon^{(1)}_n\in \mathrm{O}(\frac{1}{\sqrt{n}})$.

\textit{Verifying condition (ii) of \Cref{thm:pac-bayes2}:}
As in the previous part, $Z_0=0$, implying that $D_{\theta,\theta_0}$. Applying \Cref{prop:2} and integrating with respect to $\rho_n$, we obtain 
\begin{align}
    \int \Var\left[ \loglik \right] \rho_n(d\theta)
     &\leq \sum_{i=1}^n\left(\frac{4}{n}+2n^{\delta/2}\int C_{\theta_0,\theta}^{(i)}\rho_n(d\theta)\right)\left(\alpha_{i-1}^{\delta/(2+\delta)}\right) \nonumber \\
         & \hspace{-.5in}+ \sum_{i,j=1}^n\left(\frac{4}{n}+2n^{\delta/2}(\int C_{\theta_0,\theta}^{(i)}\rho_n(d\theta)+\int C_{\theta_0,\theta}^{(j)}\rho_n(d\theta)+\int \sqrt{C_{\theta_0,\theta}^{(i)}C_{\theta_0,\theta}^{(j)}}\rho_n(d\theta))\right)  \left(\alpha_{|i-j|-1}^{\delta/(2+\delta)}\right). 
\end{align}
First, consider the term $\int C_{\theta_0,\theta}^{(i)}\rho_n(d\theta)$. Using \Cref{assume:gen-lip}, we can upper bound $C_{\theta_0,\theta}^{(i)}$ as,
\begin{align*}
    C_{\theta_0,\theta}^{(i)} & \leq \E\left[\sum_{k=1}^{m} M^{(1)}_k(X_i,X_{i-1})|f^{(1)}_k(\theta,\theta_0)|\right]^{2+\delta}\\
       & \leq \sum_{k=1}^{m} m^{1+\delta} \E\left[\left( M^{(1)}_k(X_i,X_{i-1})|f^{(1)}_k(\theta,\theta_0)|\right)^{2+\delta}\right] \text{(by Jensen's inequality)}\\
    & =\sum_{k=1}^{m} m^{1+\delta} \E\left[ M^{(1)}_k(X_i,X_{i-1})^{2+\delta}\right]|f^{(1)}_k(\theta,\theta_0)|^{2+\delta}.
    \intertext{Since $M^{(1)}_k$'s are upper bounded by $Q$, it follows that, $C_{\theta_0,\theta}^{(i)}  \leq \sum_{k=1}^{m} m^{1+\delta} Q^{2+\delta}|f^{(1)}_k(\theta,\theta_0)|^{2+\delta}$.}
    \intertext{Hence, from \Cref{assume:gen-lip}, we get,}
    \int C_{\theta_0,\theta}^{(i)} \rho_n(d\theta) & \leq  \sum_{k=1}^{m} m^{1+\delta} Q^{2+\delta} \int |f^{(1)}_k(\theta,\theta_0)|^{2+\delta}\rho_n(d\theta)
   \leq  (mQ)^{2+\delta} \frac{C}{n}.
\end{align*}
Using the upper bound above, we can say for an $L$ large enough, $\int C_{\theta_0,\theta}^{(i)} \rho_n(d\theta)  \leq \frac{L}{n}$. 
\sloppy
Next, by the Cauchy-Schwarz inequality, we have that $\int \sqrt{C_{\theta_0,\theta}^{(i)}C_{\theta_0,\theta}^{(j)}\rho_n(d\theta))}<\sqrt{\int C_{\theta_0,\theta}^{(i)}\rho_n(d\theta)\int C_{\theta_0,\theta}^{(j)}\rho_n(d\theta))}\leq \frac{L}{n}$. Thus, we have the following upper bound.
\begin{align*}
    \int \Var\left[ \loglik \right] \rho_n(d\theta) & \leq
    \sum_{i=1}^n\left(\frac{4}{n}+2n^{\delta/2}\frac{L}{n}\right)\left(\alpha_{i-1}^{\delta/(2+\delta)}\right) + 
    \sum_{i,j=1}^n\left(\frac{4}{n}+2n^{\delta/2}(\frac{L}{n}+\frac{L}{n}+\frac{L}{n})\right)  \left(\alpha_{|i-j|-1}^{\delta/(2+\delta)}\right)\\\nonumber
          & =  \left(\frac{4}{n}+2n^{\delta/2}\frac{L}{n}\right)\left(\sum_{i=1}^n\alpha_{i-1}^{\delta/(2+\delta)}\right) +
          \left(\frac{4}{n}+6n^{\delta/2}\frac{L}{n}\right)\left(\sum_{i,j=1}^n \alpha_{|i-j|-1}^{\delta/(2+\delta)}\right).
\end{align*}
Since $\sum_{i,j=1}^n \alpha_{|i-j|-1}^{\delta/(2+\delta)}<n\sum_{k\geq1}\alpha_{k-1}^{\delta/(2+\delta)}<\infty$, we have that for some $\epsilon_n^{(2)}\in \mathrm{O}(\frac{n^{\delta/2}}{n})$, 
\begin{align*}
    \int \Var\left[ \loglik \right] \rho_n(d\theta) < n\epsilon_n^{(2)}.
\end{align*}
Since $\mathcal{K}(\rho_n,\pi)\leq \sqrt{n}C$, following the concluding argument in \Cref{thm:lip-gen} completes the proof.
{
}
\end{proof}

\subsection{\Cref{thm:bd-ns}}

\begin{proof}[Proof of \Cref{thm:bd-ns}]

We verify \Cref{assume:gen-lip} and the proof follows from \Cref{thm:bound-func}. For $i\in \{1,2,\dots,K-1\}$, 
\begin{align*}
    p_{\theta}(j|i) = \begin{cases}
    \theta &~\text{if}~j=i-1,\\
    1-\theta&~\text{if}~j=i+1.
    \end{cases}
\end{align*}
If $i=0$ or $i=K$, then the Markov chain goes back to $1$ or $K-1$ respectively with probability 1. With the convention $\log\frac{0}{0}=0$, the log ratio of the transition probabilities becomes,
\begin{align*}
    | \log p_{\theta_0}(X_1|X_0)-\log p_{\theta}(X_1|X_0) | = I_{[X_1=X_0+1]}\log\left(\frac{\theta_0}{\theta}\right)+I_{[X_1=X_0-1]}\log\left(\frac{1-\theta_0}{1-\theta}\right).
\end{align*}
In this case, $m=2$. $M^{(1)}_1(X_1,X_0)=I_{[X_1=X_0+1]}$ and $M^{(1)}_2(X_1,X_0)=I_{[X_1=X_0-1]}$, both of which are bounded. 
Let $f^{(1)}_1(\theta,\theta_0) := \log\left(\frac{\theta_0}{\theta}\right)$ suppose  $f^{(1)}_2(\theta,\theta_0) :=\log\left(\frac{1-\theta_0}{1-\theta}\right)$. 

The stationary distribution  $q_{\theta}(i)=\frac{1}{K}\enspace\forall\enspace i\in {1,2,\dots,K}$. Hence the log of the ratio of the invariant distribution becomes
\begin{align}
    \log q_{0}(x)-\log q_{\theta}(x) & = 0,
\end{align}
and we can set $M^{(2)}_i(\cdot):=1$ and $f^{(2)}_i(\cdot,\cdot):=0$ for $i\in \{1,2\}$. Thus, to prove the concentration bound for this Markov chain it is enough to assume that $\delta=1$ and show that $\int [f^{(1)}_1(\theta,\theta_0)]^3 \rho_n(d\theta)<\frac{C}{n}$ and $\int [f^{(1)}_2(\theta,\theta_0)]^3 \rho_n(d\theta)<\frac{C}{n}$ for some constant $C>0$.

As given, $\{\rho_n\}$ is a sequence of beta probability distribution functions, with parameters $\alpha_n, \beta_n$ that satisfy the constraint $\frac{\alpha_n}{\alpha_n+\beta_n}=\theta_0$. Specifically, we choose $\alpha_n = n\theta_0$ and (therefore) $\beta_n = n(1-\theta_0)$. Thus, we get the following,
\begin{align*}
    \int |f^{(1)}_1(\theta,\theta_0)|^3 \rho_n(d\theta) & = \int \left|\log\left(\frac{\theta_0}{\theta}\right)\right|^3 \rho_n(d\theta)\\
    & < \int \left|\frac{\theta_0}{\theta}-1\right|^3 \rho_n(d\theta)\\
    & = \frac{1}{\text{Beta}(\alpha_n,\beta_n)} \int_0^1 \left|\frac{ \theta_0-\theta } { \theta }\right|^{3} \theta^{\alpha_n-1} (1-\theta)^{\beta_n -1}d\theta.\\
    \intertext{Since $\theta_0,\theta\in (0,1)$, so is $\frac{|\theta_0-\theta|}{2}$, giving $|\theta_0-\theta|^3<2(\theta_0-\theta)^2$. We use that fact to arrive at}
    \int |f^{(1)}_1(\theta,\theta_0)|^3 \rho_n(d\theta) & \leq \frac{2}{\text{Beta}(\alpha_n,\beta_n)} \int_0^1 ( \theta_0-\theta)^2\theta^{\alpha_n-3} (1-\theta)^{\beta_n -1}d\theta\\
    & = \frac{2\text{Beta}(\alpha_n-3,\beta_n)}{\text{Beta}(\alpha_n,\beta_n)} \frac{(\alpha_n-3)(\beta_n)}{(\alpha_n+\beta_n-3)^2(\alpha_n+\beta_n-2)}.
\end{align*}
From our choice of $\alpha_n$ and $\beta_n$, $\frac{2\text{Beta}(\alpha_n-3,\beta_n)}{\text{Beta}(\alpha_n,\beta_n)}=O(1)$, and plugging the values of $\alpha_n$ and $\beta_n$ into $\frac{(\alpha_n-3)(\beta_n)}{(\alpha_n+\beta_n-3)^2(\alpha_n+\beta_n-2)}$, we get $\frac{(\alpha_n-3)(\beta_n)}{(\alpha_n+\beta_n-3)^2(\alpha_n+\beta_n-2)}=\frac{1}{n}\frac{(\theta_0-\frac{3}{n})(1-\theta_0)}{(1-\frac{3}{n})^2(1-\frac{2}{n})}$, which is upper bounded by $\frac{C_1}{n}$ for some constant $C_1>0$. Hence,
\begin{align*}
     \int |f^{(1)}_1(\theta,\theta_0)|^3 \rho_n(d\theta) & <\frac{C_1}{n}.
\end{align*}
Similarly, we can also show that,
\begin{align*}
     \int |f^{(1)}_2(\theta,\theta_0)|^3 \rho_n(d\theta) & < \frac{C_2}{n}.
\end{align*}
Finally, from \Cref{lem:kl-prior}, we get that $\mathcal{K}(\rho_n,\pi)<C+\frac{1}{2}\log(n)$ for some large constant $C$. Hence, $\mathcal{K}(\rho_n,\pi)<C_3\sqrt{n}$ for some constant $C_3>0$.
Choosing $C=\max(C_1,C_2,C_3)$, we satisfy all the conditions of \Cref{assume:gen-lip} and \Cref{thm:bound-func}. 
\end{proof}

%
%
%
%
%

\subsection{Proof of \Cref{thm:bd2-ns}}

\begin{proof}
For the purpose of this proof, we choose $\rho_n$'s with scaled Beta distribution with parameters $\alpha_n=n(2\theta_0)$ and $\beta_n=n(1-2\theta_0)$. Since, $\rho_n$ is a scaled Beta distribution with the scaling factors $m=0.5$ and $c=0$, the pdf of $\rho_n$ is given by
\begin{align*}
    \rho_n(\theta) & = \frac{0.5}{\text{Beta}(\alpha_n,\beta_n)}\left(2\theta\right)^{\alpha_n}\left(1-2\theta\right)^{\beta_n}
\end{align*} 
Since this is a scaled distribution, $E_{\rho_n}[\theta]=0.5\frac{\alpha_n}{\alpha_n+\beta_n}=\theta_0$ and there exists a constant $\sigma>0$, $\Var_{\rho_n}[\theta]=\frac{\sigma^2}{n}$. Now, we analyse the transition probabilities.
For $i \in \{1,2,\dots\}$, the Birth-Death process has transition probabilities
\begin{align*}
    p_{\theta}(j|i) = \begin{cases}
    \theta &~\text{if}~j=i-1,\\
    1-\theta&~\text{if}~j=i+1.
    \end{cases}
\end{align*}
If $i=0$, then the Markov chain goes to $1$ with probability $1$. Hence with the convention $\log\frac{0}{0}=0$ the ratio of the log of the transition probabilities becomes,
\begin{align*}
    | \log p_{\theta_0}(X_1|X_0)-\log p_{\theta}(X_1|X_0) | = I_{[X_1=X_0+1]}\log\left[\frac{\theta_0}{\theta}\right]+I_{[X_1=X_0-1]}\log\left[\frac{1-\theta_0}{1-\theta}\right].
\end{align*}
In this case, $m=3$. $M^{(1)}_1(X_1,X_0)=I_{[X_1=X_0+1]}$ and $M^{(1)}_2(X_1,X_0)=I_{[X_1=X_0-1]}$. Define $M^{(1)}_3(X_1,X_0) := 1$. All these random variables are bounded. Define $f^{(1)}_1(\theta,\theta_0) := \log\left[\frac{\theta_0}{\theta}\right], f^{(1)}_2(\theta,\theta_0) :=\log\left[\frac{1-\theta_0}{1-\theta}\right]$ and $f^{(1)}_3(\theta,\theta_0):=0$. Similarly as in the proof on \Cref{thm:bd-ns}, 
\begin{align*}
     \int [f^{(1)}_1(\theta,\theta_0)]^3 \rho_n(d\theta) & <\frac{C_1}{n}, \text{ and }\\
     \int [f^{(1)}_2(\theta,\theta_0)]^3 \rho_n(d\theta) & < \frac{C_2}{n}.
\end{align*}
The stationary distribution is given by $q_{\theta}(i)=(\frac{\theta}{1-\theta})^{i-1} q_{\theta}(1)\enspace\forall\enspace i\in {1,2,\dots}$, so that $q_{\theta}(i)= (1-\theta)(\frac{\theta}{1-\theta})^{i-1}$  Hence the log of the ratio of the invariant distribution becomes
\begin{align}
    \log q_{0}(i)-\log q_{\theta}(i) & = \log\left[\frac{1-\theta_0}{1-\theta}\right]+(i-1)\log\left[\frac{\theta_0}{\theta}\right]-(i-1)\log\left[\frac{1-\theta_0}{1-\theta}\right]
\end{align}
We define $M^{(2)}_1(X_0):=1$, and $M^{(2)}_2(X_0)=M^{(2)}_3(X_0):=X_0-1$. We can write $\E_{q^{(0)}} [M^{(2)}_2(X_0)]^2 = \sum_{i=1}^{\infty} (i-1)^2 q^{(0)}(i)< \sum_{i=1}^{\infty} i^2 q^{(0)}(i)$. We have chosen $q^{(0)}$ such that $\sum_{i=1}^{\infty} i^2 q^{(0)}(i)$ is bounded. Hence, $\E_{q^{(0)}} [M^{(2)}_2(X_0)]^2 <\infty$. To verify \Cref{assume:gen-lip(0)} define,
$f^{(2)}_1(\theta,\theta_0)=-f^{(2)}_3(\theta,\theta_0):=\log\left[\frac{1-\theta_0}{1-\theta}\right]$, and define $f^{(2)}_2(\theta,\theta_0):=\log\left[\frac{\theta_0}{\theta}\right]$. Therefore following the proof of \Cref{thm:bd-ns},
\begin{align*}
    \int |f^{(2)}_1(\theta,\theta_0)|^3 \rho_n(d\theta)=\int |f^{(2)}_3(\theta,\theta_0)|^3 \rho_n(d\theta) = & \int |f^{(1)}_2(\theta,\theta_0)|^3\rho_n(d\theta)<\frac{C_2}{n}, \text{ and },\\ 
    \int |f^{(2)}_2(\theta,\theta_0)|^3 \rho_n(d\theta) = & \int |f^{(1)}_1(\theta,\theta_0)|^3 \rho_n(d\theta)<\frac{C_1}{n}.
\end{align*}
Finally, we take the KL-divergence $\mathcal{K}(\rho_n,\pi)$. $\rho_n$ follows a scaled Beta distribution on $(0,1/2)$ with parameters $\alpha_n=n(2\theta_0)$ and $\beta_n=n(1-2\theta_0)$, while $\pi$ follows a scaled Beta distribution on $(0,1/2)$ with parameters $\alpha$ and $\beta$. Thus,
\begin{align*}
    \mathcal{K}(\rho_n,\pi) & = \int_0^{\frac{1}{2}} \frac{\rho_n(\theta)}{\pi(\theta)} \rho_n(d\theta),\\
    \intertext{which, by substituting $t=2\theta$, we get,}
    \mathcal{K}(\rho_n,\pi) & = 2\int_0^{1} \frac{\rho_n(t)}{\pi(t)} \rho_n(dt).
\end{align*}
$\int_0^{1} \frac{\rho_n(t)}{\pi(t)} \rho_n(dt)$ is the KL-divergence between a Beta distribution with parameters $\alpha_n$ and $\beta_n$ and a Beta distribution with parameters $\alpha$ and $\beta$. An application of \cref{lem:kl-prior} gives us for a constant $C_1>0$,
\begin{align*}
    \int_0^{1} \frac{\rho_n(t)}{\pi(t)} \rho_n(dt) < C_1+\frac{1}{2}\log(n).
\end{align*} 
Hence we can say, $  \mathcal{K}(\rho_n,\pi) <2\left[C_1+\frac{1}{2}\log(n)\right]$. Thus, we now get that for some constant $C_3>0$,
\begin{align*}
    \mathcal{K}(\rho_n,\pi) <C_3\sqrt{n}.
\end{align*}
Choosing $C=\max(C_1,C_2,C_3)$ we satisfy all of the conditions of \Cref{assume:gen-lip} and thus by \Cref{thm:bound-func}, we are complete the proof.
\end{proof}



\subsection{Proof of \Cref{thm:lip-gen-ns} : }
\begin{proof}
\textit{Verification of condition (i) of \Cref{thm:pac-bayes2}}
As in the proof of \Cref{thm:lip-gen} substitute the true parameter $\theta_0$  for $\theta_1$ and $\theta$ for $\theta_2$. 
We also set $q_{1}^{(0)}$ and $q_{2}^{(0)}$ to the known initial distribution $q^{(0)}$. 
Similar to the steps leading to \cref{equation:kld-upper-bound-ns}, we get
\begin{align*}
    \kld & \leq \sum_{i=1}^n  \sum_{k=1}^{m} \E \left[ M^{(1)}_k(X_i,X_{i-1}) \right]|f^{(1)}_k(\theta,\theta_0)|.
\end{align*}
Consider the term $\E\left[{ M^{(1)}_k(X_i,X_{i-1})} \right]$. 
With $q^{(i-1)}_{\theta_0}$ the marginal distribution of $X_{i-1}$, we have
\begin{align*}
    \E\left[ M^{(1)}_k(X_i,X_{i-1}) \right] & = \int M^{(1)}_k(x_i,x_{i-1}) p_{\theta_0}(x_{i}|x_{i-1}) q^{(i-1)}_{\theta_0}(x_{i-1}) dx_{i}dx_{i-1}.\\
\E\left[ M^{(1)}_k(X_i,X_{i-1}) \right]  & =  \int M^{(1)}_k(x_i,x_{i-1}) p_{\theta_0}(x_{i}|x_{i-1}) p_{\theta_0}^{i-1}(x_{i-1}|x_{0}) q^{(0)}_{\theta_0}(x_0)dx_{0} dx_{i}dx_{i-1}\\
\intertext{ Recall that the marginal density satisfies $q^{(i-1)}_{\theta_0}(x_{i-1})=\int  p_{\theta_0}^{i-1}(x_{i-1}|x_{0}) q^{(0)}_{\theta_0}(x_0)d(x_{0})$, where $ p_{\theta_0}^{i}(\cdot|x_{0})$ is the $i$-step transition probability. Then }
  \E\left[ M^{(1)}_k(X_i,X_{i-1}) \right]  & =  \int \E \left[M^{(1)}_k(X_i,x_{i-1})|x_{i-1}\right] p_{\theta_0}^{i-1}(x_{i-1}|x_{0}) q^{(0)}_{\theta_0}(x_0)dx_{0} dx_{i-1}.
\end{align*}
 Since the Markov chain $\{X_n\}$ satisfies \Cref{assume:drift}, we know by the application of \Cref{theorem:v-geom} that $\{X_n\}$ is $V$-geometrically ergodic. Hence, $\exists \enspace \tau<1$, $R<\infty$  such that $\forall \enspace |f|<V$
\begin{align*}
    |\int f(x_{i-1}) p_{\theta_0}^{i-1}(x_{i-1}|x_{0})  dx_{i-1} & -\int f(x_{i-1}) q_{\theta_0}(x_{i-1})dx_{i-1}|  < RV(x_0)\tau^{i-1},\\
\intertext{where $q_{\theta_0}$ is the stationary distribution, implying that}
    \int f(x_{i-1}) p_{\theta_0}^{i-1}(x_{i-1}|x_{0})  dx_{i-1} & < \int f(x_{i-1}) q_{\theta_0}(x_{i-1})dx_{i-1}+RV(x_0)\tau^{i-1}.
\end{align*}
Using Jensen's inequality we have $\left(\E\left[ M^{(1)}_k(X_i,X_{i-1})|X_{i-1} \right]\right)^{2+\delta} \leq \E\left[ M^{(1)}_k(X_i,X_{i-1})^{2+\delta}|X_{i-1} \right]<V(X_{i-1})$. Since $V(\cdot) \geq 1$, it follows that  $\E\left[ M^{(1)}_k(X_i,X_{i-1})|X_{i-1} \right]<V(X_{i-1})^{1/(2+\delta)} \leq V(X_{i-1})$.
Thus, setting $f(x) = \E\left[M^{(1)}_k(X_i,X_{i-1})|X_{i-1} = x\right]$, we obtain
\begin{align*}
     \E\left[ M^{(1)}_k(X_i,X_{i-1}) \right] & <\int\left[  \E\left[M^{(1)}_k(X_i,X_{i-1})|X_{i-1}\right] q_{\theta_0}(x_i)dx_{i-1}+RV(x_0)\tau^{i-1}\right]q^{(0)}(x_0)dx_0\\
     & =  \E[M^{(1)}_k(X_1,X_0)]+\tau^{i-1}\int RV(x_0)q^{(0)}(x_0)dx_{0}.\\
\intertext{Summing from $i=1$ to $n$, we get }
    \sum_{i=1}^n \E\left[ M^{(1)}_k(x_i,x_{i-1}) \right]& < n\E[M^{(1)}_k(X_1,X_0)]+\sum_{i=1}^n \tau^{i-1}\int RV(x_0)q^{(0)}(x_0)dx_{0} \\
     & =  n\E[M^{(1)}_k(X_1,X_0)]+\frac{1-\tau^n}{1-\tau}\int RV(x_0)q^{(0)}(x_0)dx_{0}.
\end{align*}
This gives us the following bound on $\int \kld \rho_n(d\theta)$:
\begin{align*}
     \int \kld \rho_n(d\theta) & \leq \sum_{k=1}^m \left[n\E[M^{(1)}_k(X_1,X_0)]+\frac{1-\tau^n}{1-\tau}\int RV(x_0)D(x_0)dx_{0}\right]\int |f^{(1)}_k(\theta,\theta_0)|\rho_n(d\theta).
\end{align*}
By \Cref{assume:gen-lip}, $\int |f^{(1)}_k(\theta,\theta_0)|\rho_n(d\theta)<\frac{C}{\sqrt{n}}$. Hence, we can rewrite the previous expression as
\begin{align*}
    \int \kld \rho_n(d\theta) & \leq \sum_{k=1}^m \left[n\E[M^{(1)}_k(X_1,X_0)]+\frac{1-\tau^n}{1-\tau}\int RV(x_1)D(x_1)dx_{1}\right]\frac{C}{\sqrt{n}}\\
    & = nm\left[\E[M^{(1)}_k(X_1,X_0)]+\frac{1-\tau^n}{n(1-\tau)}\int RV(x_0)D(x_0)dx_{0}\right]\frac{C}{\sqrt{n}}.
    \intertext{Since, $\tau<1$, $0<1-\tau^n<1$, and we rewrite the previous equation as,}
    \int \kld \rho_n(d\theta) &  \leq nm \left[\E[M^{(1)}_k(X_1,X_0)] + \frac{1}{n(1-\tau)}\int RV(x_0)D(x_0)dx_{0}\right]\frac{C}{\sqrt{n}}.
    \intertext{Hence, there exists an $\epsilon^{(1)}_n\in \mathrm{O}(\frac{1}{\sqrt{n}})$ such that $\int \kld \rho_n(d\theta)  \leq n\epsilon^{(1)}_n$.}
\end{align*}
\textit{Verification of condition (ii) of \Cref{thm:pac-bayes2}:} Similar to as in the proof of \Cref{thm:bound-func}, we upper bound $\int \Var\left[ \loglik \right] \rho_n(d\theta)$ by
\begin{align}
     \int \Var\left[ \loglik \right] \rho_n(d\theta) & \leq \sum_{i,j=1}^n\left(\frac{4}{n}+2n^{\delta/2}(\int C_{\theta_0,\theta}^{(i)}\rho_n(d\theta)+\int C_{\theta_0,\theta}^{(j)}\rho_n(d\theta)+\int \sqrt{C_{\theta_0,\theta}^{(i)}C_{\theta_0,\theta}^{(j)}}\rho_n(d\theta))\right)  \left(\alpha_{|i-j|-1}^{\delta/(2+\delta)}\right)\\\nonumber
     &\qquad+ \sum_{i=1}^n\left(\frac{4}{n}+2n^{\delta/2}\int C_{\theta_0,\theta}^{(i)}\rho_n(d\theta)\right)\left(\alpha_{i-1}^{\delta/(2+\delta)}\right),
\end{align}
where $C_{\theta_0,\theta}$ is upper bounded as
\begin{align*}
    C_{\theta_0,\theta}^{(i)} & \leq \sum_{k=1}^{m} m^{1+\delta} \E\left[ M^{(1)}_k(X_i,X_{i-1})\right]^{2+\delta}|f^{(1)}_k(\theta,\theta_0)|^{2+\delta}.
    \intertext{Since $\E\left[ M^{(1)}_k(X_i,X_{i-1})^{2+\delta}|X_{i-1} \right]<V(X_{i-1})$, by a similar application of $V$-geometric ergodicity, we can say that, $\exists\enspace 0<\tau<1$, such that }
     \E\left[ M^{(1)}_k(X_i,X_{i-1}) \right]^{2+\delta}& \leq  n\E[M^{(1)}_k(X_1,X_0)]^{2+\delta}+ \tau^{i-1}\int RV(x_0)D(x_0)dx_{0}, 
     \intertext{which, by the fact that $\tau^{i-1}<\tau$, gives us,}
     \E\left[ M^{(1)}_k(X_i,X_{i-1}) \right]^{2+\delta} & \leq  \E[M^{(1)}_k(X_1,X_0)]^{2+\delta}+\tau\int RV(x_0)D(x_0)dx_{0}.
\end{align*}
By \Cref{assume:gen-lip}, we know that, $\int |f^{(1)}_k(\theta,\theta_0)|^{2+\delta}\rho_n(d\theta)<\frac{C}{n}$. Hence, for a large constant $L$, $\int C_{\theta_0,\theta}^{(i)} \rho_n(d\theta) \leq \frac{L}{n}$. We also see that since the chain is geometrically ergodic, by the application of \cref{eq:alpha_mixing-coeff-sum}, $\sum_{k \geq 1} \alpha_k^{\delta/(2+\delta)} < +\infty$. The rest of the proof follows similarly as in the proof of \Cref{thm:bound-func}, and we obtain an $\epsilon^{(2)}_n\in \mathrm{O}(\frac{n^{\delta/2}}{n})$, such that,
\begin{align*}
    \int \Var[\loglik]\rho_n(d\theta) < n\epsilon^{(2)}_n.
\end{align*}
Since, $\mathcal{K}(\rho_n,\pi)\leq \sqrt{n}C$, similar arguments as in the proof of \Cref{thm:lip-gen} holds. The theorem is thus proved.

\end{proof}

\subsection{Proof of \Cref{thm:gen-nonstationary-hajek}}

\begin{proof}
\textit{Verification of condition (i) of \Cref{thm:pac-bayes2}}
As in the proof of \Cref{thm:lip-gen} substitute the true parameter $\theta_0$  for $\theta_1$ and $\theta$ for $\theta_2$. We also set our initial distributions  $q_1^{(0)}$ and $q_2^{(0)}$ to the known initial distribution $q^{(0)}$. A method similar to \Cref{equation:kld-upper-bound-ns}, yields
\begin{align*}
    \kld & \leq \sum_{i=1}^n  \sum_{k=1}^{m} \E \left[ M^{(1)}_k(X_i,X_{i-1}) \right]|f^{(1)}_k(\theta,\theta_0)|.
\end{align*}
Because $M^{(1)}_k$s satisfy \Cref{assume:hajek}, it follows by the application of Theorem 2.3, \cite{hajek} that $\exists \enspace \lambda>0$ such that for any $0<\kappa\leq \lambda$, and for some $\zeta \in (0,1)$ possibly depending upon $\lambda$, 
\begin{align*}\E\left[e^{\kappa M^{(1)}_k(X_i,X_{i-1})}\right|X_1,X_0] \leq \zeta^{i-1}e^{\kappa M^{(1)}_k(X_1,X_{0})}+\frac{1-\zeta^i}{1-\zeta}\mathcal{D}e^{\kappa a} \quad \text{ for all } i>1.
\end{align*}
We rewrite $\E \left[ M^{(1)}_k(X_i,X_{i-1}) |X_1,X_0\right]$ as follows:
\begin{align*}
    \E \left[ M^{(1)}_k(X_i,X_{i-1}) |X_1,X_0 \right] & =  \frac{\E[ \kappa M^{(1)}_k(X_i,X_{i-1})|X_1,X_0]}{\kappa}\\
    & \leq \frac{\E[e^{\kappa M^{(1)}_k(X_i,X_{i-1})}|X_1,X_0]}{\kappa}.
\end{align*}
Therefore, $\sum_{i=1}^n\E \left[ M^{(1)}_k(X_i,X_{i-1}) \right]$ can be upper bounded as, 
\begin{align*}
    \sum_{i=1}^n\E \left[ M^{(1)}_k(X_i,X_{i-1}) \right] & =\sum_{i=1}^n\E \left[ M^{(1)}_k(X_i,X_{i-1})|X_1,X_0 \right]\kappa^{-1}\\
    &\leq\sum_{i=1}^n  \left[ \zeta^{i-1}\E e^{\kappa M^{(1)}_k(X_1,X_{0})}+\frac{1-\zeta^i}{1-\zeta}\mathcal{D}e^{\kappa a}\right]\kappa^{-1}.
    \intertext{Since, $\zeta\in (0,1)$, $\zeta^i<1$. Hence, we can write that,}
    \sum_{i=1}^n  \left[ \zeta^{i-1}\E e^{\kappa M^{(1)}_k(X_1,X_{0})}+\frac{1-\zeta^i}{1-\zeta}\mathcal{D}e^{\kappa a}\right] & \leq \sum_{i=1}^{n}  \left[ \zeta^{i-1}\E e^{\kappa M^{(1)}_k(X_1,X_{0})}+\frac{1}{1-\zeta}\mathcal{D}e^{\kappa a}\right]\kappa^{-1}\\
    & = \left[ \frac{1-\zeta^n}{1-\zeta}\E e^{\kappa M^{(1)}_k(X_1,X_{0})}+\frac{n}{1-\zeta}\mathcal{D}e^{\kappa a}\right]\kappa^{-1}\\
    & \leq nL,
\end{align*}
for a large constant $L$. Therefore $\int \kld \rho_n(d\theta)$ can be upper bounded as follows,
\begin{align*}
    \int \kld \rho_n(d\theta) & \leq \int \sum_{k=1}^m  nL|f^{(1)}_k(\theta,\theta_0)| \rho_n(d\theta)\\
    & = \sum_{k=1}^m  nL \int |f^{(1)}_k(\theta,\theta_0)| \rho_n(d\theta).
    \intertext{By \Cref{assume:gen-lip}, $\int |f^{(1)}_k(\theta,\theta_0)| \rho_n(d\theta)<\frac{C}{n}$, hence,  }
    \int \kld \rho_n(d\theta) & \leq  n L \frac{C}{\sqrt{n}}.
\end{align*}
Hence, for some $\epsilon_n^{(1)}\in \mathrm{O}(\frac{1}{\sqrt{n}})$, we have obtained that, $\int \kld \rho_n(d\theta)\leq n\epsilon_n^{(1)}$.

\textit{Verification of condition (ii) of \Cref{thm:pac-bayes2}:} Similar to as in the proof of \Cref{thm:bound-func}, we upper bound $\int \Var\left[ \loglik \right] \rho_n(d\theta)$ by
\begin{align}
     \int \Var\left[ \loglik \right] \rho_n(d\theta) & \leq \sum_{i,j=1}^n\left(\frac{4}{n}+2n^{\delta/2}(\int C_{\theta_0,\theta}^{(i)}\rho_n(d\theta)+\int C_{\theta_0,\theta}^{(j)}\rho_n(d\theta)+\int \sqrt{C_{\theta_0,\theta}^{(i)}C_{\theta_0,\theta}^{(j)}}\rho_n(d\theta))\right)  \left(\alpha_{|i-j|-1}^{\delta/(2+\delta)}\right)\\\nonumber
     &\qquad+ \sum_{i=1}^n\left(\frac{4}{n}+2n^{\delta/2}\int C_{\theta_0,\theta}^{(i)}\rho_n(d\theta)\right)\left(\alpha_{i-1}^{\delta/(2+\delta)}\right),
\end{align}
where $C_{\theta_0,\theta}$ is upper bounded as
\begin{align*}
    C_{\theta_0,\theta}^{(i)} & \leq \sum_{k=1}^{m} m^{1+\delta} \E\left[ M^{(1)}_k(X_i,X_{i-1})\right]^{2+\delta}|f^{(1)}_k(\theta,\theta_0)|^{2+\delta}.
\end{align*}
There exists a constant $C_{\delta}$ depending upon $\delta$ such that,
\begin{align*}
[M^{(1)}_k]^{2+\delta}(X_i,X_{i-1}) & = \frac{\kappa^{2+\delta}[M^{(1)}_k]^{2+\delta}(X_i,X_{i-1})^{2+\delta}}{\kappa^{2+\delta}}\\
& \leq \frac{ e^{\kappa M^{(1)}_k(X_i,X_{i-1})}+C_{\delta}}{\kappa^{2+\delta}}.
\end{align*}
By expressing $ \E \left[ M^{(1)}_k(X_i,X_{i-1})^{2+\delta} \right] = \E\left[ \E \left[ M^{(1)}_k(X_i,X_{i-1})^{2+\delta} |X_1,X_0 \right] \right] $ and following a method similar to the previous part, we get,
\begin{align*}
      \E \left[ M^{(1)}_k(X_i,X_{i-1})^{2+\delta} \right] & \leq \frac{\left[ \zeta^i\E e^{\kappa M^{(1)}_k(X_1,X_{0})}+\frac{1-\zeta^i}{1-\zeta}\mathcal{D}e^{\kappa a}\right] +C_{\delta}}{\kappa^{2+\delta}}.
     \intertext{The fact that $0<\zeta<1$ implies that $0<\zeta^i<\zeta$. This gives us the following,}
     \E \left[ M^{(1)}_k(X_i,X_{i-1})^{2+\delta} \right] & \leq \frac{\left[ \zeta\E e^{\kappa M^{(1)}_k(X_1,X_{0})}+\frac{1}{1-\zeta}\mathcal{D}e^{\kappa a}\right] +C_{\delta}}{\kappa^{2+\delta}}.
     \intertext{Since $\kappa<\lambda$, by the application of Jensen's inequality, we get}
     \E \left[ M^{(1)}_k(X_i,X_{i-1})^{2+\delta} \right] & \leq \frac{\left[ \zeta\E e^{\lambda M^{(1)}_k(X_1,X_{0})}+\frac{1}{1-\zeta}\mathcal{D}e^{\kappa a}\right] +C_{\delta}}{\kappa^{2+\delta}}\\
     & = \frac{\left[ \zeta\int e^{\lambda M^{(1)}_k(x_1,x_{0})}p_{\theta_0}(x_1|x_0)D(x_0)dx_1dx_0+\frac{1}{1-\zeta}\mathcal{D}e^{\kappa a}\right] +C_{\delta}}{\kappa^{2+\delta}}.
\end{align*}
We know that $\int |f^{(1)}_k(\theta,\theta_0)|^{2+\delta}\rho_n(d\theta)<\frac{C}{n}$. Thus, following \Cref{assume:gen-lip} we can say that, for a large constant $L$, $\int C_{\theta_0,\theta}^{(i)} \rho_n(d\theta) \leq \frac{L}{n}$. The rest of the proof follows similarly as in the proof of \Cref{thm:bound-func}, and we obtain an $\epsilon^{(2)}_n\in \mathrm{O}(\frac{n^{\delta/2}}{n})$, such that,
\begin{align*}
    \int \Var[\loglik]\rho_n(d\theta) < n\epsilon^{(2)}_n.
\end{align*}
Since, $\mathcal{K}(\rho_n,\pi)\leq \sqrt{n}C$, similar arguments as in the proof of \Cref{thm:lip-gen} holds. The theorem is thus proved. 
\end{proof}

\subsection{Proof of \Cref{ex:slm}}
\begin{proof}
For the purpose of the proof, we choose $\rho_n$'s with scaled Beta distribution with parameters $\alpha_n=n\frac{1+\theta_0}{2}$ and $\beta_n=n\frac{1-\theta_0}{2}$. Since, $\rho_n$ is a scaled Beta distribution with the scaling factors $m=2$ and $c=-1$, the pdf of $\rho_n$ is given by
\begin{align*}
    \rho_n(\theta) & = \frac{2}{\text{Beta}(\alpha_n,\beta_n)}\left(\frac{1+\theta}{2}\right)^{\alpha_n}\left(\frac{1-\theta}{2}\right)^{\beta_n}
\end{align*} 
Since this is a scaled distribution, $E_{\rho_n}[\theta]=2\frac{\alpha_n}{\alpha_n+\beta_n}-1=\theta_0$ and there exists a constant $\sigma>0$, $\Var_{\rho_n}[\theta]=\frac{\sigma^2}{n}$.
We now analyse the log-ratio of the transition probabilities for the Markov chain,
\begin{align*}
 \log p_{\theta_0}(X_n|X_{n-1})-\log p_{\theta}(X_n|X_{n-1})  = 2X_nX_{n-1}(\theta-\theta_0)+X_{n-1}^2(\theta_0^2-\theta^2).\\
\end{align*}
Observe that in this setting, $M^{(1)}_1(X_n,X_{n-1})=|X_nX_{n-1}|$ and $M^{(1)}_2(X_n,X_{n-1})=X_n^2$. 
Next, using the fact that
\begin{align*}
\E[|X_n|^{2+\delta}|X_{n-1}] & = \E[|X_n-\theta_0X_{n-1}+\theta_0X_{n-1}|^{2+\delta}|X_{n-1}], \\
\intertext{and by an application of triangle inequality, we obtain}
    \E[|X_n|^{2+\delta}|X_{n-1}] & \leq \E \left[\left(|X_n-\theta_0X_{n-1}|+|\theta_0X_{n-1}|\right)^{2+\delta}|X_{n-1}\right]\\
    & = \E \left[\left(2 \frac{|X_n-\theta_0X_{n-1}|+|\theta_0X_{n-1}|}{2}\right)^{2+\delta}|X_{n-1}\right]\\
    & = \E \left[2^{2+\delta}\left( \frac{|X_n-\theta_0X_{n-1}|+|\theta_0X_{n-1}|}{2}\right)^{2+\delta}|X_{n-1}\right].\\
\intertext{Now by using Jensen's inequality we get,}
     \E[|X_n|^{2+\delta}|X_{n-1}] & \leq \E \left[2^{2+\delta}\left( \frac{|X_n-\theta_0X_{n-1}|^{2+\delta}+|\theta_0X_{n-1}|^{2+\delta}}{2}\right)|X_{n-1}\right]\\
    & = 2^{1+\delta}\E\left[|X_n-\theta_0X_{n-1}|^{2+\delta}|X_{n-1}\right]+2^{1+\delta}|\theta_0X_{n-1}|.\\
\end{align*}
We know if $Y\sim N(\mu,\sigma^2)$, then $
\E|Y-\mu|^p  =\sigma^p\frac{2^{\frac{p}{2} \Gamma(\frac{p+1}{2})}}{\sqrt{\pi}}$. Consequently,
\begin{align}~\label{equation:expectation-normal-power}
    \E[|X_n|^{2+\delta}|X_{n-1}] & \leq 2^{1+\delta}\left[\frac{2^{\frac{2+\delta}{2} \Gamma(\frac{3+\delta}{2})}}{\sqrt{\pi}}\right]+2^{1+\delta}|\theta_0X_{n-1}|^{2+\delta}.
\end{align}
It follows that,
\begin{align*}
    \E[M^{(1)}_1(X_n,X_{n-1})^{2+\delta}|X_{n-1}] & \leq  2^{1+\delta}\left[\frac{2^{\frac{2+\delta}{2} \Gamma(\frac{3+\delta}{2})}}{\sqrt{\pi}}\right]|X_{n-1}|^{2+\delta}+2^{1+\delta}|\theta_0|^{2+\delta}|X_{n-1}|^{4+2\delta}\\
    & \leq \left(2^{1+\delta}\left[\frac{2^{\frac{2+\delta}{2} \Gamma(\frac{3+\delta}{2})}}{\sqrt{\pi}}\right]+2^{1+\delta}|\theta_0|^{2+\delta}\right)(|X_{n-1}|^{4+2\delta}+1).\\
\intertext{Since $\theta_0<1$, we can say,}
    \E[M^{(1)}_1(X_n,X_{n-1})^{2+\delta}|X_{n-1}] & \leq \left(2^{1+\delta}\left[\frac{2^{\frac{2+\delta}{2} \Gamma(\frac{3+\delta}{2})}}{\sqrt{\pi}}\right]+2^{1+\delta}\right)(|X_{n-1}|^{4+2\delta}+1).\\
\intertext{Define a constant $C_{\delta}:= \left(2^{1+\delta}\left[\frac{2^{\frac{2+\delta}{2} \Gamma(\frac{3+\delta}{2})}}{\sqrt{\pi}}\right]+2^{1+\delta}\right)$. The above term then becomes,}
   \E[M^{(1)}_1(X_n,X_{n-1})^{2+\delta}|X_{n-1}] & \leq  C_{\delta}(|X_{n-1}|^{4+2\delta}+1).\\
\end{align*}
Next we analyse the term $M^{(1)}_2(X_n,X_{n-1})$.
\begin{align*}
    \E\left[ M^{(1)}_2(X_n,X_{n-1})^{2+\delta} |X_{n-1}\right] & = \E[X_{n-1}^{4+2\delta}|X_{n-1}]\\
    & = X_{n-1}^{4+2\delta}\\
    & \leq C_{\delta}(X_{n-1}^{4+2\delta}+1).
\end{align*}
Then, defining $V(x):=C_{\delta}(x^{4+2\delta}+1)$ it follows that,
\begin{align*}
    \E \left[V(X_n) | X_{n-1}\right] & = \E \left[C_{\delta}(X_{n}^{4+2\delta}+1) | X_{n-1}\right]. \\
\intertext{Using a technique similar to \Cref{equation:expectation-normal-power} we get,}
     \E \left[C_{\delta}(X_{n}^{4+2\delta}+1) | X_{n-1}\right] & \leq \left[C_{\delta}( 2^{3+2\delta}\left[\frac{2^{\frac{4+2\delta}{2} \Gamma(\frac{5+2\delta}{2})}}{\sqrt{\pi}}\right]+2^{3+2\delta}|\theta_0X_{n-1}|^{4+2\delta}+1) \right].\\
\end{align*}
Define another constant $C_{\delta}':= C_{\delta}\left(2^{3+2\delta}\left[\frac{2^{\frac{4+2\delta}{2} \Gamma(\frac{5+2\delta}{2})}}{\sqrt{\pi}}\right]-2^{3+2\delta}|\theta_0|^{4+2\delta}+1\right)$. Since $\delta>0$, $\frac{2^{\frac{4+2\delta}{2} \Gamma(\frac{5+2\delta}{2})}}{\sqrt{\pi}}>1$. And since $|\theta_0|<1$, so is $|\theta_0|^{4+2\delta}$. Hence,

\begin{equation*}
    2^{3+2\delta}\left[\frac{2^{\frac{4+2\delta}{2} \Gamma(\frac{5+2\delta}{2})}}{\sqrt{\pi}}\right]-2^{3+2\delta}|\theta_0|^{4+2\delta}> 0.
\end{equation*}
Hence, we have shown that,
\begin{align*}
     \E \left[V(X_n) | X_{n-1}\right] & \leq (2^{3+2\delta}|\theta_0|^{4+2\delta})C_{\delta}(X_{n-1}^{4+2\delta}+1)+C_{\delta}'.
     \intertext{Since $|\theta_0|<2^{\frac{1}{4+2\delta}-1}$, $2^{3+2\delta}|\theta_0|^{4+2\delta}<1$, and we can express the above equation as, }
     \E \left[V(X_n) | X_{n-1}\right] & \leq V(X_{n-1})+C_{\delta}'.
\end{align*}
Define the set $C(m) := \{x:|x|^{4+2\delta}+1\leq m\}$. From Proposition 11.4.2, \cite{meyntweedie}, for a large enough $m$, $C(m)$ forms a petite set. Thus, we have proved that $V(x)$ as defined in this example satisfies \Cref{assume:drift}, and $\{X_n\}$ is $V$-geometrically ergodic.
The $f^{(1)}_j$'s corresponding to \Cref{assume:gen-lip} are given by $f^{(1)}_1(\theta,\theta_0)=(\theta-\theta_0)$ and $f^{(1)}_2(\theta,\theta_0)=(\theta_0^2-\theta^2)$. Therefore, it follows that,
\begin{align*}
    \partial_{\theta}  f^{(1)}_1 & = 1,\\
    \partial_{\theta}  f^{(1)}_2 & = -2\theta \text{ and } \\
    -2 & < -2\theta < 2.
\end{align*}
Since $f^{(1)}_{1} (\theta_0,\theta_0)=f^{(1)}_{2} (\theta_0,\theta_0)=0$, We just showed that they also have bounded partial derivatives. We also know that $|\theta|<1$. Hence, by \Cref{prop:3} $f^{(1)}_j$'s satisfy the conditions of \Cref{assume:gen-lip}.

The invariant distribution for the simple linear model Markov-chain under parameter $\theta$ is given by a gaussian distribution with mean $0$ and variance $\frac{1}{1-\theta^2}$. In other words,
\begin{align*}
    q_{\theta}(x)= \frac{1}{\sqrt{2\pi}}e^{-\frac{1-\theta^2}{2}x^2}.
\end{align*}
Analyzing the log likelihood yields,
\begin{align*}
    \log q_{0}(x)-\log q_{\theta}(x) & =-\frac{x^2}{2}(1-\theta_0^2)+\frac{x^2}{2}(1-\theta^2)\\
    & = \frac{x^2}{2}(\theta_0^2-\theta^2).
\end{align*}
Let $f^{(2)}_1(\theta_0,\theta_0)=(\theta_0^2-\theta^2)$ and $f^{(2)}_1(\theta_0,\theta_0)=0$. Since $f^{(2)}_1(\theta_0,\theta_0)=f^{(1)}_2(\theta_0,\theta_0)$, by following arguments similar as before, can conclude that $f^{(2)}_1(\theta_0,\theta_0)$ also satisfies the requirements of \Cref{assume:gen-lip}. Let $M^{(2)}_1(x)=\frac{x^2}{2}$ and define $M^{(2)}_2(x):=1$. Let $X_0\sim q_1^{(0)}$. As long as $\int x^{4+2\delta} q_1^{(0)}(x) dx<\infty$, we satisfy all the conditions required for \Cref{thm:lip-gen-ns}.
Finally we need to verify the condition that $\mathcal{K}(\rho_n,\pi)<C\sqrt{n}$ for some constant $C>0$. 
The KL-divergence $\int \log \left(\frac{\rho_n(\theta)}{\pi(\theta)}\right)\rho_n(d\theta)$ becomes,
\begin{align*}
     \mathcal{K}(\rho_n,\pi) & = \int_{-1}^1 \log\left(\frac{2}{\text{Beta}(\alpha_n,\beta_n)}\left(\frac{1+\theta}{2}\right)^{\alpha_n}\left(\frac{1-\theta}{2}\right)^{\beta_n}\right) \frac{2}{\text{Beta}(\alpha_n,\beta_n)}\left(\frac{1+\theta}{2}\right)^{\alpha_n}\left(\frac{1-\theta}{2}\right)^{\beta_n} d\theta.
\end{align*}
Substituting, $y=\frac{1+\theta}{2}$, we get,
\begin{align*}
    \mathcal{K}(\rho_n,\pi) & = \int_{0}^1 \log\left(\frac{2}{\text{Beta}(\alpha_n,\beta_n)}\left(y\right)^{\alpha_n}\left(1-y\right)^{\beta_n}\right) \frac{1}{\text{Beta}(\alpha_n,\beta_n)}\left(y\right)^{\alpha_n}\left(1-y\right)^{\beta_n} dy \\
    & = \int_0^1 \log(2) \frac{1}{\text{Beta}(\alpha_n,\beta_n)}\left(y\right)^{\alpha_n}\left(1-y\right)^{\beta_n} dy \\
    & \qquad +\int_{0}^1 \log\left(\frac{1}{\text{Beta}(\alpha_n,\beta_n)}\left(y\right)^{\alpha_n}\left(1-y\right)^{\beta_n}\right) \frac{1}{\text{Beta}(\alpha_n,\beta_n)}\left(y\right)^{\alpha_n}\left(1-y\right)^{\beta_n}. 
\end{align*}
{The first term integrates upto 2. The second term is the KL-divergence between a Uniform and Beta distribution with parameters $\alpha_n=n\frac{1+\theta_0}{2}$ and $\beta_n=n(1-\frac{1+\theta_0}{2})$ and support $[0,1]$. Following \Cref{prop:kl-prior-unif} it follows that $\mathcal{K}(\rho_n,\pi)$ is upper bounded by,}
\begin{align*}
    \mathcal{K}(\rho_n,\pi) & <2+C_1+\frac{1}{2}\log(n) < C\sqrt{n},
\end{align*}
for some large constant $C$. This completes the proof.
\end{proof}
\subsection{Proof of \Cref{thm:pac-bayes-miss}}
\begin{proof}
As in the proof of \Cref{thm:pac-bayes}, following \cref{eq:proof-step-miss}, we note that,
\begin{align}~\label{eq:proof-step-miss-fin}
    \int \malpharen \Tilde{\pi}_{n,\alpha|X^n}(d\theta) &\leq \frac{\alpha}{1-\alpha}\int \kld  \rho_n(d\theta)+\frac{\alpha}{1-\alpha} \sqrt{ \frac{\text{Var}[\int \loglik \rho_n(d\theta)]}{\eta} }\nonumber\\ & \qquad \qquad +\frac{\mathcal{K}(\rho_n,\pi)-\log(\epsilon)}{1-\alpha}.
\end{align}
Following from \cref{eq:miss-kl} and \cref{eq:miss-var}, we get that,
\begin{align*}
     \int \kld \rho_n(d\theta) & \leq \E[r_n(\theta_0,\theta^*_n)]+n\epsilon_n,
     \intertext{and}
     \int \Var[r_n(\theta,\theta_0)] \rho_n(d\theta) & \leq 2n\epsilon_n +2\Var[r_n(\theta^*_n,\theta_0)].
\end{align*}
Plugging these into \cref{eq:proof-step-miss-fin}, we are done.
\end{proof}

\bibliographystyle{plain}
\bibliography{biblio}
\end{document}